\documentclass[11pt]{amsart}
\usepackage{amsmath, amsthm, amssymb,amscd}

\usepackage{graphicx}
\usepackage{tikz}
\usepackage{geometry} 
\usepackage[arrow,curve,matrix]{xy}
\usetikzlibrary{cd,calc,arrows}
\usepackage[english]{babel}
\newcommand{\C}{\mathbb{C}}

\newcommand{\F}{\mathcal{F}}

\newtheorem{thm}{Theorem}[section]
\newtheorem*{thm-nl}{Theorem}
\newtheorem*{prop-nl}{Proposition}
\newtheorem{definition}[thm]{Definition}

\def\GG{{\textbf G}}

\def\PP{{\textbf P}}
\def\OO{\mathcal{O}}
\def\cN{\mathcal{N}}

\def\cB{\mathcal{B}
\def\H{\mathcal{H}}}
\def\cA{\mathcal{A}}

\def\F{\mathcal{F}}
\def\cW{\mathcal{W}}

\def\cP{\mathcal{P}}
\def\kk{\mathbf{k}}

\def\J{\mathcal{J}}
\def\T{\mathcal{T}}
\def\cM{\mathcal{M}}
\def\cR{\mathcal{R}}

\def\cQ{\mathcal{Q}}
\def\H{\mathcal{H}}
\def\Pic0{{\rm Pic}^0(X)}

\newcommand{\Sym}{\operatorname{Sym}}

\newtheorem{cor}[thm]{Corollary}
\newtheorem*{cor-nl}{Corollary}

\newtheorem*{conjecture-nl}{Conjecture}
\newtheorem*{quest-nl}{Question}
\newtheorem*{quests-nl}{Questions}

\newtheorem{prop}[thm]{Proposition}

\theoremstyle{definition}

\newtheorem{ex}[thm]{Example}
\newtheorem{rmk}[thm]{Remark}
\newtheorem{question}[thm]{Question}

\title{{Syzygies and Koszul modules in geometry}}

\author[G. Farkas]{Gavril Farkas}
\address{Humboldt-Universit\"at zu Berlin, Institut f\"ur Mathematik,  Unter den Linden 6
\hfill \newline\texttt{}
 \indent 10099 Berlin, Germany} \email{{\tt farkas@math.hu-berlin.de}}

\begin{document}

\begin{abstract}
We describe the progress in the last 10 years related to Koszul modules and syzygies of algebraic varieties. Topics discussed include the general theory of Koszul modules and resonance varieties, applications to Chen ranks of K\"ahler and hyperplane arrangement groups (Suciu's Conjecture) and connections related to syzygies of algebraic curves. Developments related to Green's Conjecture, the Secant Conjecture and the Gonality Conjecture on the resolution of line bundles on algebraic curves are also presented. Open question are proposed throughout the text.
\end{abstract}

\maketitle

\section{Introduction}

Koszul modules and their associated resonance varieties, indirectly first appeared 100 years ago as \emph{Alexander-type} invariants in topology. The very beginning is probably the Alexander invariant of a knot \cite{Al}. If $K\subseteq S^3$ is a knot, we denote by $X:=S^{3}\setminus K$ its complement and consider $X^{\mathrm{ab}}$ its universal abelian cover. Then the homology $H_1(X^{\mathrm{ab}}, \mathbb Z)$ is a module over the group ring $\mathbb Z\bigl[H_1(X, \mathbb Z)\bigr]\cong \mathbb Z[t, t^{-1}]$ and the Alexander polynomial of the knot $K$ can be interpreted as the order of this module. This definition can be immediately generalized to any reasonable topological space $X$, and indeed, to any finitely generated group $G$ (for a space $X$, the group in question being the fundamental group $G=\pi_1(X)$). In this case, the \emph{Alexander invariant} $B(G)$ of the group becomes a module over the group ring $\mathbb C\bigl[G_{\mathrm{ab}}\bigr]$, see \cite{Mas}, \cite{PS2}. Considering the associated graded object $\mbox{gr} B(G)$ of the Alexander invariant with respect to the powers of the augmentation ideal $I\subseteq \mathbb C\bigl[G_{\mathrm{ab}}]$, it turns out that the rank of the abelian group $\mbox{gr}_q B(G)=I^q\cdot B(G)/I^{q+1}\cdot B(G)$ is equal to the \emph{Chen rank} $\theta_{q+2}(G)$ of the group, first considered by K.-T. Chen \cite{Ch} in the context of \emph{iterated integrals} on a manifold $X$ and arising from the lower central series of the metabelian quotient of the group $G$.

\vskip 4pt

The Chen invariants $\theta_q(G)$ of a group $G$ are fundamental invariants used to distinguish in subtle ways between groups of geometric nature.  It has therefore become of great interest to develop methods to make it possible to determine these invariants. In this context,  the definition of the \emph{infinitesimal Alexander invariant} $W(G)$ of the group $G$ has been put forward by Papadima and Suciu in \cite{PS1} and \cite{PS2}.\footnote{Note that in \cite{PS1} the notation $\mathfrak{B}(G)$ has been used for the infinitesimal Alexander invariant. This notation has been abandoned in the subsequent papers \cite{AFPRW2}, \cite{AFRS1}, \cite{AFRS2}.}  This is a graded module over the symmetric algebra $\mbox{Sym } H_1(G,\mathbb C)$ and its degree $q$ part $W_q(G)$ has the property that $\mbox{dim } W_q(G)\leq \theta_{q+2}(G)$, with equality if the group is $1$-formal in the sense of Sullivan \cite{Sul}. On the other hand, $W(G)$, being a graded module over a polynomial ring, is an object firmly rooted in commutative algebra and thus a large
variety of tools from homological algebra, derived categories and algebraic geometry become suddenly available in order to study its structure. The first instance of a non-trivial calculation of Chen ranks in Magnus' calculation \cite{Mag} of the Chen ranks of the free group $F_n$, even predating Chen's work.

\vskip 2pt

In this spirit, the study of Koszul modules, and even their definition, became more algebraic. This process,  initiated in \cite{PS2}, took off in the papers \cite{AFPRW1}, \cite{AFPRW2}, \cite{AFRS1} and \cite{AFRS2}, where results concerning the support and the Hilbert series of \emph{arbitrary} Koszul modules are established. The group $G$ is replaced with a pair $(V,K)$, where $V$ is a finite dimensional vector space and $K\subseteq \bigwedge^2 V$ is a linear subspace. The Koszul module $W(V,K)$  becomes now a graded module over the symmetric algebra $S=\mbox{Sym}(V)$ and as such, it is an object belonging to mainstream algebraic geometry. The general results on Koszul modules were then applied to a large variety of groups, like K\"ahler groups, mapping class groups and Torelli groups \cite{AFPRW2}, \cite{AFRS1}, or 
hyperplane arrangement groups \cite{AFRS2}. In this paper a wide ranging generalization of \emph{Suciu's Conjecture} on the Chen ranks of a hyperplane arrangement group is put forward and proved for arbitrary Koszul modules, whose resonance satisfies a number of natural geometric conditions inspired by Hodge theory.

\vskip 4pt

One of the  surprises of the theory of Koszul modules is that the new perspective they provide could be used to  make decisive progress on major questions on syzygies of algebraic varieties, even though these questions a priori have little to do with each other. Given a projective variety $X\subseteq \PP^r$ over an algebraically closed field  $\kk$ and embedded by a line bundle $L\in \mbox{Pic}(X)$, the minimal graded resolution of the section ring 
$$\Gamma_X(L):=\bigoplus_{q\geq 0} H^0\bigl(X, L^{q}\bigr)$$
viewed as a module over the symmetric algebra $S:=\mbox{Sym } H^0(X,L)$ establishes the language in which one can talk about the \emph{structure of  equations} in modern algebraic geometry. The information contained in this resolution is packaged in the \emph{graded Betti table} of the pair $(X,L)$, by placing the quantity $$b_{p,q}(X,L)=\mbox{dim } \mbox{Tor}_S^p \bigl(\Gamma_X(L),  \kk \bigr)_{p+q},$$ that is,  the number of $p$-syzygies of weight $q$,  in the $p$-th column and the $q$-th row of the Betti table. The modern theory of syzygies originates in the fundamental work of Hilbert \cite{Hil}, whose \emph{Syzygienstaz} made it possible to even think of the Betti table of a module and then was set on solid footing by Grothendieck, Mumford and others via  sheaf cohomology and concepts like Castelnuovo-Mumford regularity. 

\vskip 4pt

For algebraic curves, the connection between the abstract geometry of the curve (in the form of its linear systems) and the structure of its Betti tables has been subject to three major conjectures formulated by Green and Lazarsfeld \cite{GL2}. These are \emph{Green's Conjecture} concerning the Betti table of an arbitrary canonical curve $X\subseteq \PP^{g-1}$, the \emph{Secant Conjecture} describing the vanishing in the second row of the Betti table of a pair $(X,L)$, where $L$ is a non-special line bundle on $X$, and the \emph{Gonality Conjecture} describing the vanishing in the first row of the Betti table of $(X,L)$ in terms of the gonality of the curve. Green's Conjecture for general curves has been famously proved by Voisin in the early 2000's in two fundamental papers \cite{V1}, \cite{V2}. This has been then used by Farkas and Kemeny \cite{FK1} to establish the Secant Conjecture for a general pair $(X,L)$. Around the same time, Ein and Lazarsfeld \cite{EL2} proved the Gonality Conjecture for an arbitrary curve $X$ and for a line bundle $L$ of sufficiently high degree. An optimal effective version of the Gonality Conjecture for a general $k$-gonal curve of genus $g$ has been proved in \cite{FK4}.  This has been extended to arbitrary $k$-gonal curves of genus $g$ in Niu and Park's remarkable recent paper \cite{NP}. 

\vskip 4pt

Despite many partial results, see for instance \cite{AF1}, \cite{Ap}, Green's Conjecture for an arbitrary curve of genus $g$ remains open. In the last few years however, we witnessed major progress on our understanding of the syzygies of general canonical curves. Using essential input from the theory of Koszul modules, it has been possible to represent the Koszul cohomology groups of a certain canonical curve of maximal Clifford index as the components of a Koszul module corresponding to a Clebsch--Gordan decomposition, and thus deduce the Generic Green's Conjecture from a general vanishing theorem for Koszul module with trivial resonance \cite{AFPRW1}. In particular a version of the Generic Green's Conjecture in not too low positive characteristic could be established. Soon after this, several other proofs of the Generic Green's Conjecture appeared, including those of Raicu--Sam using a bigraded version for Koszul modules, or Kemeny's geometric proof \cite{K1} representing a major simplification of Voisin's Theorem \cite{V1}, \cite{V2} concerning the Koszul cohomology of a polarized K3 surface of Picard number one.  

\vskip 4pt

We close this introduction by mentioning also the striking progress on the syzygies of secant varieties of curves embedded in projective space by line bundles of sufficiently high degree \cite{ENP}.

\subsection*{Acknowledgments}
This paper is dedicated to the memory of \c S. Papadima who first introduced me to Koszul modules. Without his insight the connections between syzygies and Koszul modules would probably not have been discovered. I have greatly profited from discussions with M.~Aprodu, M.~Kemeny, R. Lazarsfeld, J. Park, C. Raicu, F.-O. Schreyer, A. Suciu, C.~Voisin and J. Weyman related to this circle of ideas. I am particularly grateful to M. Aprodu and C. Raicu for many useful comments on the first version of this paper. This work was supported by the Berlin Mathematics Research Center MATH+ and by the ERC Advanced Grant SYZYGY.  
    
\section{Koszul modules}
We discuss the algebraic definition of Koszul modules, largely following \cite{AFPRW1}, \cite{AFPRW2}, \cite{AFRS1} and \cite{AFRS2}. We fix  
an $n$-dimensional vector space $V$ over an algebraically closed field $\kk$ and set $S:=\mbox{Sym}(V)$. 
We fix throughout a linear subspace $K\subseteq \bigwedge^2 V$.

\vskip 4pt

We consider the Koszul differential
\[
\delta_p\colon \bigwedge^pV\otimes S\longrightarrow \bigwedge^{p-1}V\otimes S,
\]
\[
\delta_p(v_1 \wedge \cdots \wedge v_p \otimes f) = \sum_{j=1}^p (-1)^{j-1} v_1 \wedge \cdots \wedge \widehat{v_j} \wedge \cdots \wedge v_p \otimes v_j\cdot f.
\]
We have a decomposition $\delta_p=\bigoplus_q\delta_{p,q}$, where
$\delta_{p,q}\colon \bigwedge^p V \otimes \mbox{Sym}^q V \to \bigwedge^{p-1}V \otimes \mbox{Sym}^{q+1}V.$ The Koszul complex putting together all the differentials $\delta_p$  provides a minimal resolution by free graded $S$-modules of the residue field $\kk$. We denote by 
$$\pi \colon \bigwedge^3V \otimes S(-1)\longrightarrow \Bigl(\bigwedge^2V/K\Bigr)\otimes S$$ 
the morphism of graded $S$-modules induced after projection by the third differential $\delta_3$ in the Koszul complex. 

\begin{definition}\label{def:koszul_module}
The \emph{Koszul module} of the pair $(V,K)$ is the graded $S$-module defined as $W(V,K):=\mathrm{coker}(\pi)$.
\end{definition}
Equivalently, $W(V,K)$ is the graded 
$S$-module admitting the following presentation:
\begin{equation}
    \label{eqn:presentation_W}
    0\longrightarrow \bigwedge^3V\otimes S(-1)\stackrel{\pi} \longrightarrow \Bigl(\bigwedge^2V/K\Bigr)\otimes S 
    \longrightarrow W(V,K) \longrightarrow 0.
    \end{equation}

By using that the Koszul complex is exact, we can also represent the degree $q$ component $W_q(V,K)$ of the Koszul module as the following $\kk$-vector space
\begin{equation}\label{eq:def-Wq}
W_q(V,K)=\mbox{middle homology} \Bigl\{K\otimes \mathrm{Sym}^q V \stackrel{\delta_2}\longrightarrow
 V\otimes \mathrm{Sym}^{q+1} V\stackrel{\delta_1}\longrightarrow \mathrm{Sym}^{q+2} V  
\Bigr\}.
\end{equation}
From the exactness of the Koszul complex, one has $W_q\bigl(V, \bigwedge^2 V)=0$. 

\vskip 4pt

Via the Euler sequence on the projective space $\PP:=\PP\bigl(V^{\vee}\bigr)$
$$
0\longrightarrow \Omega_{\PP} \longrightarrow V\otimes \mathcal{O}_{\PP}(-1)\longrightarrow 
\mathcal{O}_{\PP}\longrightarrow 0,$$
we find, after taking cohomology of a twist, that 
$W_q(V,0)\cong H^0\bigl(\PP,\Omega_{\PP}(q+2)\bigr)$.
In particular, 
\begin{equation}\label{eq:Koszul_zero}
\dim W_q(V,0)=(q+1)\binom{n+q}{2+q}.
\end{equation}

The definition of a Koszul module is  functorial. An inclusion 
$K\subseteq K'\subseteq \bigwedge^2V$ induces a surjection at the level of Koszul modules $W(V,K)\twoheadrightarrow W(V,K')$. Also, it is easy to prove that if $W_q(V,K)=0$, then $W_{q'}(V,K)=0$ for all $q'\geq q$.

\vskip 4pt

The presentation (\ref{eqn:presentation_W}) and the description (\ref{eq:def-Wq}) can be sheafified, and we can introduce the \emph{Koszul sheaf}
$\cW(V,K)$ on $\PP$, defined by the following exact sequence
\begin{equation}\label{eq:Koszul_sheaf}
K\otimes \OO_{\PP}\stackrel{\delta_2}\longrightarrow \Omega_{\PP}(2)\longrightarrow \cW(V,K)\longrightarrow 0.
\end{equation}

\subsection{The BGG correspondence and Koszul modules} Even though this is not apparent from the papers \cite{AFPRW1}, \cite{AFPRW2}, \cite{AFRS2}, the definition of the Koszul module $W(V,K)$ in inspired by the \emph{BGG correspondence}. Recall that the Bernstein--Gelfand--Gelfand correspondence establishes an equivalence between the category of graded $E:=\bigwedge^{\bullet} V^{\vee}$-modules and the category of linear free complexes of graded $S=\mbox{Sym}(V)$-modules. To a graded $E$-module $P=\bigoplus_{i=0}^n P_i$, one associates a complex of graded $S$-modules
$$\mathbb L(P): \cdots \longrightarrow S\otimes_{\kk} P_{i}\longrightarrow S\otimes_{\kk} P_{i-1}\longrightarrow \cdots.$$
If $\widehat{P}$ is the dual $E$-module defined by $\widehat{P}_i:=(P_{-i})^{\vee}$, then one has an identification between $H_k\bigl(\mathbb L(P)\bigr)_{i+k}$ and $\mbox{Tor}_i^{E}\bigl(\widehat{P}, \kk\bigr)_{-i-k}$, see \cite[Theorem 7.8]{Eis2}.
One applies this correspondence to the following \emph{quadratic algebra} 
$$A(K):=E/\mathrm{ideal}(K^{\perp}),$$
where $K^{\perp}\subseteq \bigwedge^2 V^{\vee}$ can be identified with $\bigl(\bigwedge^2 V/K\bigr)^{\vee}$. Then one has the following isomorphism 
\begin{equation}\label{eq:BGG}
W_q(V,K)^{\vee}\cong \mbox{Tor}_{q+1}^{E}\bigl(A(K), \kk\bigr)_{-q-2},
\end{equation}  
see also \cite[Proposition 2.8]{PS2}. In particular, any result on Koszul modules can be rephrased in terms of resolutions of quadratic algebras over the exterior algebra $E$.  

\subsection{Resonance varieties} The support of the Koszul module $W(V,K)$ turns out to be a familiar object of algebro-geometric nature, namely the \emph{resonance variety} $\mathcal{R}(V,K)$ defined as the locus 
\begin{equation}
\label{eq:def-resonance}
\mathcal{R}(V,K)=\Bigl\{a\in V^\vee : \text{ there exists $b\in
V^\vee$ such that $a\wedge b\in K^\perp\setminus \{0\}$} \Bigr\}\cup \{0\}.
\end{equation}
The following fact was first observed in \cite[Lemma 2.4]{PS2}:

\begin{prop}
One has the set-theoretic equality $\mathrm{supp }\  W(V,K)=\cR(V,K)$
\end{prop}
\begin{proof} 
Using the description (\ref{eq:Koszul_sheaf}), we observe that $\mbox{supp } W(V,K)$ is the affine cone over the sheaf-theoretic support of the Koszul sheaf $\cW(V,K)$. This can be described as the locus of those points $[a]\in \PP=\PP(V^{\vee})$ such that the fibre over $[a]$ of the morphism $\delta_2\colon K\otimes \OO_{\PP}\rightarrow \Omega_{\PP}(2)$ is not surjective. But this map is the contraction by $a$, given by $v_1\wedge v_2\mapsto a(v_1)v_2-a(v_2)v_1$, which in turn, can be viewed as the dual of the map $\wedge a\colon V^{\vee}\rightarrow K^{\vee}$. Therefore, $[a]\in \mbox{supp } \cW(V,K)$ if an only if the map $\wedge a$ is not injective, that is, $a\in\cR(V,K)$. 
\end{proof}

\subsection{Koszul modules with vanishing resonance}

The Koszul modules of finite length are precisely those for which the resonance variety vanishes. Thus $W_q(V,K)=0$ for $q\gg 0$ if and only if $\mathcal{R}(V,K)=\{0\}$. The paper \cite{AFPRW2} came out of an attempt to find an optimal constant $q$ for which the vanishing above should hold. This question is already posed in \cite{PS2}.

\vskip 3pt

By identifying the Grassmannian $\mathbf{G}:=\mbox{Gr}_2(V^{\vee})$ of $2$-dimensional quotients of $V$ in its Pl\"ucker embedding in $\PP\bigl(\bigwedge^2 V^{\vee}\bigr)$ with the space of rank $1$ tensors, one has the following equivalence:
\begin{equation}\label{eq:van_res}
\cR(V,K)=0\Longleftrightarrow \mathbf G\cap \PP\bigl(K^{\perp}\bigr)=\emptyset.
\end{equation} 

Since $\mbox{dim}(\mathbf{G})=2n-4$, for dimension reasons, if $\cR(V,K)=\{0\}$, then necessarily $$\mbox{dim}(K)\geq \mbox{dim}(\mathbf G)+1= 2n-3.$$ The borderline case $\mbox{dim}(K)=2n-3$ turns out to be particularly relevant. We quote the main result of \cite{AFPRW2}:

\begin{thm}\label{thm:va-resonance}
Let $V$ be a $\kk$-vector space of dimension $n\geq 3$ and let $K\subseteq\bigwedge^2 V$ be a subspace. Suppose $\mathrm{char}(\kk)=0$, or $\mathrm{char}(\kk)\geq n-2$. We have the following equivalence:
\begin{equation}\label{eq:thm-main-equiv}
\mathcal{R}(V,K) = \{0\} \Longleftrightarrow W_{n-3}(V,K) = 0.
\end{equation}
\end{thm}

This result is established in \cite{AFPRW2} in characteristic zero and in \cite[Theorem 1.3]{AFPRW1} when $\mbox{char}(\kk)\geq n-2$. The strategy of the proof relies on regarding the $m$-dimensional subspace $K$ as a subspace of section of $H^0\bigl(\mathbf G, \OO_{\mathbf G}(1)\bigr)\cong \bigwedge^2 V$ and  interpreting the condition  $\cR(V,K)=\{0\}$ as saying that the evaluation map $\mathrm{ev}_K\colon K\otimes\mathcal{O}_\mathbf{G}\to \mathcal{O}_\mathbf{G}(1)$ is surjective on $\mathbf{G}$.

\vskip 4pt

Denoting by $\cQ$ the tautological rank $2$ vector bundle on $\mathbf{G}$, via the description (\ref{eq:def-Wq}), we observe that the have the identification $$H^0\bigl(\PP, \Omega_{\PP}(q+2)\bigr)\cong H^0\bigl(\GG, \mbox{Sym}^q\cQ(1)\bigr).$$
The vanishing $W_{n-3}(V,K)=0$ is then equivalent to the surjectivity of the multiplication map 
$$K\otimes H^0(\mathbf{G}, \mbox{Sym}^{n-3} \cQ)\longrightarrow H^0\bigl(\mathbf{G}, \mbox{Sym}^{n-3} \cQ(1)\bigr).$$ 
 
To the surjective morphism of sheaves $\mathrm{ev}_K$ we attach the Koszul complex 
\begin{equation}
\label{eqn:Koszul_finite}
\mathcal{K}^\bullet:\, 0\longrightarrow\bigwedge^mK\otimes \mathcal{O}_\mathbf{G}(1-m)\longrightarrow
\cdots \longrightarrow \bigwedge^2K\otimes\mathcal{O}_\mathbf{G}(-1)\longrightarrow K\otimes \mathcal{O}_\mathbf{G}
\stackrel{\mathrm{ev}_K}\longrightarrow \mathcal{O}_\mathbf{G}(1)\longrightarrow 0.
\end{equation}
 Twisting this complex by $\mbox{Sym}^q\cQ$ and taking hypercohomology, we obtain a  
spectral sequence 
\[
E_1^{-i,j}=\bigwedge^iK\otimes H^j\bigl(\mathbf{G}, \mbox{Sym}^q \cQ(1-i)\bigr)
\]
which abuts to zero. In this language, we have the following  identification 
$$W_{n-3}(V,K)=\mbox{Coker}\bigl\{E_1^{-1,0}\longrightarrow E_1^{0,0}\bigr\}.$$ Assuming by contradiction this map is not surjective, there must exist a non-zero differential in this spectral sequence. But, as explained in \cite{AFPRW2} (or in \cite{AFPRW1} in the case of positive characteristic) this is ruled out by the Bott vanishing theorem.

\begin{rmk}
In the case $\mbox{char}(\kk)$ is arbitrary, the optimal vanishing result is due to Raicu--VandeBogert \cite[Theorem 7.1]{RS}, showing that $W_{2n-7}(V,K)=0$, whenever $\cR(V,K)=0$.
\end{rmk}

\subsection{The Chow form of the Grassmannian}
The information contained in Theorem \ref{thm:va-resonance} in the borderline case $\mbox{dim}(K)=2n-3$ can be reinterpreted in geometric terms using the \emph{Chow form} of the Grassmannian 
$\mathbf{G}\subseteq \PP\bigl(\bigwedge^2 V\bigr)$. Precisely, one has two divisors on the parameter space $\mathrm{Gr}_{2n-3}\bigl(\bigwedge ^2V\bigr)$ of such subspaces $K$. The first one is the \emph{Koszul divisor}
\[
\mathcal{D}_\mathfrak{Kosz}:=\Bigl\{K\in\mathrm{Gr}_{2n-3}\bigl(\bigwedge ^2V\bigr): W_{n-3}(V,K)\ne 0\Bigr\}, 
\]
the second one is the \emph{resonance divisor}, that is,
\[
\mathcal{D}_\mathfrak{Res}:=\Bigl\{K\in\mathrm{Gr}_{2n-3}\bigl(\bigwedge ^2V\bigr): \PP(K^\perp)\cap \mathbf{G}\neq\emptyset\Bigr\},
\]
which can be regarded as the Chow form of $\mathbf{G}$. Theorem \ref{thm:va-resonance} can be formulated as the following set-theoretic equality
$\mathrm{Supp}(\mathcal{D}_\mathfrak{Res})=\mathrm{Supp}(\mathcal{D}_\mathfrak{Kosz})$. A more precise statement can be found in \cite[Theorem 3.4]{AFRW}, where one shows the following equality of divisors on $\mathrm{Gr}_{2n-3}\bigl(\bigwedge ^2V\bigr)$:
\begin{equation}
\label{eq:class-D2}
\mathcal{D}_\mathfrak{Kosz}=(n-2)\cdot \mathcal{D}_\mathfrak{Res}.
\end{equation}
The reason for the multiplicity $n-2$ in (\ref{eq:class-D2}) is fairly transparent. It is shown in \cite[Lemma 3.4]{AFRW} that for any such $K$ with $\cR(V,K)\neq \{0\}$, one has $\mbox{dim } W_{n-3}(V,K)\geq n-2$. It follows that $\mathcal{D}_{\mathfrak{Kosz}}-(n-2)\cdot \mathcal{D}_{\mathfrak{Res}}$ is an effective divisor and the equality (\ref{eq:class-D2}) shows that this divisor is in fact trivial.

\begin{rmk}
The equality of divisors (\ref{eq:class-D2}) is reminiscent of an equality of two effective divisors on the moduli space $\cM_{2n-3}$ of algebraic curves of genus $2n-3$. Precisely, denoting by $$\mathfrak{Syz}:=\Bigl\{[C]\in \cM_{2n-3}: K_{n-2,1}(C, \omega_C)\neq 0\Bigr\}$$ the Koszul divisor of curves with a non-trivial $(n-2)$-nd syzygy of weight one and by $$\mathfrak{Hur}:=\Bigl\{[C]\in \cM_{2n-3}: W^1_{n-1}(C)\neq \emptyset\Bigr\}$$ the Hurwitz divisors of curves having submaximal gonality, then one has the equality of effective divisors $\mathfrak{Syz}=(n-2)\cdot \mathfrak{Hur}$. The multiplicity $n-2$ in this equality can be explained via the Green-Lazarsfeld Non-Vanishing Theorem \cite{Gr}, which shows in this case that for each curve $[C]\in \mathfrak{Hur}$ one has in fact that $\mbox{dim } K_{n-2,1}(C, \omega_C)\geq n-2$. This fact has been first observed by Hirschowitz--Ramaman \cite{HR} and it plays a crucial role in Aprodu's sufficient condition \cite{Ap} for a curve to satisfy Green's Conjecture. It would be highly interesting to find a direct connection between this equality of divisors on $\cM_{2n-3}$  and the equality of divisors (\ref{eq:class-D2}) on the Grassmannian. 
 \end{rmk}

\subsection{Koszul modules of geometric nature} Even though Koszul modules originate from topology, their algebraic definition opens the way to consider situations arising in natural algebro-geometric contexts. Two such important instances are discussed in \cite{AFRW}. 

\vskip 3pt 
Let $E$ be a vector bundle on an algebraic variety $X$ and consider the determinant map
$$d\colon \bigwedge^2 H^0(X,E) \longrightarrow H^0\bigl(X, \bigwedge^2 E\bigr).$$
Setting $V:=H^0(X,E)^{\vee}$ and $K:=\mbox{Ker}(d)^{\perp}\subseteq \bigwedge^2 V$, we obtain the Koszul module 
$$W(X,E):=W(V,K)$$
associated to $E$. We now explain the information contained in this Koszul module in the case $E$ is globally generated. We define the  kernel bundle $M_E$ by the exact sequence on $X$
\begin{equation}\label{eq:kernel_bundle}
0 \longrightarrow M_E\longrightarrow H^0(X,E)\otimes \OO_X\stackrel{\mathrm{ev}}\longrightarrow E\longrightarrow 0.
\end{equation}
We have the following result, see \cite[Theorem 4.3]{AFRW}:

\begin{thm}
\label{thm:det}
Let  $E$ be a globally generated vector bundle on $X$. The components of the Koszul module associated to $E$ admit the following description
\[
W_q(X,E)^\vee \cong\mathrm{Ker}\Bigl\{H^1(X,\Sym^{q+2}M_E)\longrightarrow \Sym^{q+2}H^0(X,E)\otimes H^1(X,\mathcal O_X)\Bigr\}.
\]
In particular, if $H^1(X,\mathcal O_X)=0$, then $W_q(X,E)^\vee \cong H^1\bigl(X,\Sym^{q+2}M_E\bigr)$.
\end{thm}
 
One can describe in geometric terms when the resonance $\cR(X,E)$ associated to the Koszul module $W(X,E)$ vanishes. For instance, if $E$ has rank two, then $\cR(X,E)=0$ if and only if there exist two independents sections $s_1, s_2\in H^0(X,E)$ which are proportional in each fibre of $E$. This happens if and only if $E$ admits a line subbundle $A\hookrightarrow E$ with $h^0(X,A)\geq 2$. In the case of curves, the components of the resonance $\cR(X,E)$ are described in detail in \cite[\S6]{AFRS1}.  

\vskip 4pt

\begin{question}
When $X$ is a $K3$ surface and $E$ is a Lazarsfeld-Mukai vector bundle on $X$, then from Theorem \ref{thm:det}, we have  $W^q(X,E)\cong H^1(X, \Sym^{q+2} M_E)^{\vee}$. In \cite{AFRW} one finds a description in terms of the lattice $\mbox{Pic}(X)$ when the corresponding resonance variety $\cR(X,E)$ vanishes. Can one understand the geometric implications of the vanishing provided by Theorem \ref{thm:va-resonance} in the context of Koszul modules on $K3$ surfaces?
\end{question}

An instance of Koszul modules where (at least in characteristic zero) the resonance is always trivial, is provided by \emph{Gaussian Koszul modules}. Let $L$ be a line bundle on a smooth projective variety $X$ and consider the Gaussian map
$$\psi_L \colon \bigwedge^2 H^0(X,L)\longrightarrow H^0(X,\Omega_X \otimes L^2), \ \ \mbox{ } \ \mbox{  }  \psi_L(f\wedge g)=f dg-g df.$$
When $L$ is very ample, Wahl \cite{W} related the cokernel of $\psi_L$ to the deformations of the cone over the embedded variety $X\stackrel{|L|}\hookrightarrow  \PP^r$ inside $\PP^{r+1}$. In particular, the corank of the Gaussian map of the canonical bundle can be used to distinguish whether a curve lies on a $K3$ surface. Arbarello--Bruno--Sernesi \cite{ABS} showed that the non-surjectivity of $\psi_{\omega_X}$ essentially characterizes those smooth curves that can be embedded in a $K3$ surface.

\vskip 4pt

Setting  $V:=H^0(X,L)^{\vee}$ and $K:=\mbox{Ker}(\psi_L)^{\perp}$, we obtain the Gaussian Koszul module
$$\mathcal{G}(X,L):=W(V,K).$$
The components of the Gaussian Koszul module have a description similar to that in Theorem \ref{thm:det} involving this time the bundle $R_L$ defined by the exact sequence 
$$
0 \longrightarrow  R_L \longrightarrow  H^0(X,L)\otimes \OO_X \stackrel{\mathfrak{Ta}}\longrightarrow J_1(L)\longrightarrow 0,
$$
where $J_1(L)$ is the jet bundle of $L$ and $\mathfrak{Ta}$ is the \emph{Taylor map} which locally assigns to each section of $L$ its value and its derivative at each point. It is shown in \cite[Theorem 5.2]{AFRW} that the components of the Gaussian module $\mathcal{G}(X,L)$ are given by
 \[
 \mathcal{G}_q(X,L)^{\vee} = \mbox{Ker}\Bigl\{H^1\bigl(X,\Sym^{q+2}\ R_L\bigr) \longrightarrow \Sym^{q+1} H^0(X, L) \otimes H^1(X, R_L)\Bigr\}.
 \] 
Applications of the Theorem \ref{thm:va-resonance} via Gaussian Koszul modules to the stabilization of the cohomology of thickenings of projective varieties are presented in \cite[Theorem 1.6]{AFRW}. 

\section{The Chen ranks of Koszul modules}

The very recent paper \cite{AFRS2} presents a major generalization of Theorem \ref{thm:va-resonance} and deals with Koszul modules with non-vanishing resonance. The resonance of Koszul modules of geometric or topological origin rarely vanishes, therefore all the components $W_q(V,K)$ of the corresponding Koszul module are non-zero. In spite of this, one is interested in computing their dimensions, that is, in determining the Hilbert series of the Koszul module $W(V,K)$. Inspired by \emph{Suciu's Conjecture} on hyperplane arrangement groups, an optimal solution to this problem has been put forward in \cite{AFRS2}.

\vskip 3pt

It turns out that the set-theoretic definition (\ref{eq:def-resonance}) of the resonance  $\cR(V,K)$ is not refined enough to determine the Hilbert series of $W(V,K)$ and, instead, one has to take into account the natural scheme structure of $\cR(V,K)$. The following definition appears first in \cite{AFRW} and has then been put to use in \cite{AFRS1} and \cite{AFRS2}. 

\begin{definition}\label{def:isotropic}
Let  $K\subseteq \bigwedge^2 V$ be a linear subspace as above. We say that,
\begin{enumerate}
\item the resonance $\cR(V,K)$ is  \emph{linear} if it is a union of linear subspace 
$$
\cR(V, K)=\overline{V}_1^{\vee}\cup \cdots \cup \overline{V}_k^{\vee},
$$
where each $\overline{V}_t^{\vee}\subseteq V^{\vee}$  corresponds to a quotient $V\twoheadrightarrow \overline{V}_t$.
\item The resonance $\cR(V,K)$ is \emph{isotropic}, if it is linear and 
$\bigwedge^2 \overline{V}_t^{\vee}\subseteq K^{\perp}$, for  $t=1, \ldots, k$.
\item The resonance $\cR(V,K)$ is \emph{separable}, if  
$$
\bigl(\overline{V}_t^{\vee}\wedge V^{\vee}\bigr) \cap K^{\perp}\subseteq 
\bigwedge^2 \overline{V}_t^{\vee}, \ \mbox{ for } t=1, \ldots, k.
$$
\item The resonance $\cR(V,K)$ is  \emph{strongly isotropic} is it is both isotropic and separable.
\end{enumerate}
\end{definition}

\begin{rmk}
Of these definitions, the one that requires getting used to is the one of separability. For instance, if $V^{\vee}=\langle e_1,e_2,e_3,e_4\rangle$ and one chooses $K^{\perp}=\langle e_1\wedge e_2, e_1\wedge e_3+e_2\wedge e_4\rangle \subseteq \bigwedge^2 V^{\vee}$, then the resonance $\cR(V,K)=\langle e_1,e_2\rangle=:\overline{V}^{\vee}$ is linear. The element $e_1\wedge e_3+e_2\wedge e_4\in K^{\perp}\cap \bigl(\overline{V}^{\vee}\wedge \overline{V}\bigr)$ does not lie however in $\bigwedge^2 \overline{V}^{\vee}$. Therefore, the resonance is not separable.
\end{rmk}

\begin{rmk}
The condition that a component $\overline{V}^{\vee}\subseteq V^{\vee}$ of the resonance $\cR(V,K)$ is separable can be expressed in a \emph{Petri like} condition. Indeed, if $\overline{n}=\mbox{dim}(\overline{V}^{\vee})$, we set $U:=\mbox{Ker} \bigl\{\pi\colon V\to \overline{V}\bigr\}$. We then fix a basis
$(e_1,\ldots,e_n)$ of $V^{\vee}$ such that
$(e_1,\ldots,e_{\overline{n}})$ is a basis for $\overline{V}^{\vee}$. Letting
$(v_1,\ldots,v_n)$ denote the dual basis of $V$,
we obtain a decomposition
$$
\bigwedge^2\, V = L \oplus M \oplus H, 
$$
where $L:= \mbox{span}\bigl\{ v_s \wedge v_t : s,t\leq \overline{n}\bigr\} \cong \bigwedge^2 \overline{V}$, 
$M:= \mbox{span}\bigl\{ v_s \wedge v_t : s\leq \overline{n} \mbox{ and }  t>\overline{n}\bigr\} \cong \overline{V}\otimes U$, and 
$H:= \mbox{span}\bigl\{ v_s \wedge v_t : s,t>\overline{n}\bigr\} \cong \bigwedge^2U$. With this notation in place, it is shown 
in \cite[\S3.1]{AFRS1} that an isotropic component $\overline{V}^{\vee}$ is separable of 
$\cR(V,K)$ if and only if the projection map
$$p_M\colon K\longrightarrow M$$ is surjective. In \cite{AFRS1} a more general condition in the absence of the isotropicity assumption is also provided. 
These conditions are then made explicit in the case of Koszul modules associated to vector bundles on curves, where surprisingly, the condition that the resonance be strongly isotropic is intimately connected with the notion of \emph{very stability} appearing in the study of the Hitchin system, see \cite{PPN} and references therein. 
\end{rmk}

\subsection{The scheme structure of the resonance}
One considers the annihilator of the Koszul module $I(V,K):=\mbox{Ann } W(V,K)\subseteq S$ and defines the \emph{projective resonance scheme}, by setting 
\[
{\bf{R}}(V,K):=\mbox{Proj}\bigl(S/I(V,K)\bigr).
\]
Note that the projective resonance ${\bf{R}}(V,K)$ is naturally a subscheme of the projective space $\PP=\PP(V^{\vee})$.   
We introduce the following linear section of the Grassmannian $\mathbf{G}=\mathrm{Gr}_2(V^{\vee})$
\begin{equation}\label{eq:linearsection}
\mathbf{B}(V,K):= \mathbf{G} \cap \mathbf{P}K^\perp,
\end{equation}
which is regarded as the base locus of the linear 
system $|K|$ on $\mathbf{G}$, where $K\subseteq H^0\bigl(\mathbf{G}, \mathcal{O}_{\mathbf{G}}(1)\bigr)$.

\vskip 4pt

The relation between $\mathbf{R}(V,K)$ and $\mathbf{B}(V,K)$ is explained by the  incidence correspondence

\[
\xymatrixcolsep{5pc}
\xymatrix{
\Xi \ar[r]^{\pi_2} \ar[d]^{\pi_1}& \mathbf{G}\\
\PP &
}
\]
where $\Xi=\bigl\{(p,L)\in \PP\times \mathbf{G}: p\in L\bigr\}\subseteq \PP\times \mathbf{G}$ can be regarded as the projectivization of the tautological rank $2$ quotient bundle $Q$ over $\mathbf{G}$. 

\vskip 3pt

One has the following set-theoretic equality established in \cite[Lemma 2.5]{AFRS1}
$$\mathbf{R}(V, K)=\pi_1\bigl(\pi_2^{-1}(\mathbf{B}(V,K))\bigr).$$
By definition,  ${\bf{R}}(V,K)$ is always uniruled, hence it cannot be an arbitrary subscheme of $\PP$. 

\begin{question}
Can one find a characterization of \emph{all} those subschemes of the projective space $\PP$ which appear as projective resonance schemes ${\bf{R}}(V,K)$?
\end{question}

\vskip 4pt

Theorem 1.1 from \cite{AFRS1} explains how the property of the scheme  ${\bf{R}}(V,K)$ being reduced can be translated in linear algebra terms by the Definition \ref{def:isotropic}. Precisely, one has the following:
\begin{thm}\label{thm:reduced}
Let $K\subseteq \bigwedge^2 V$ be as before and suppose that the resonance $\cR(V,K)$ is linear.
\begin{enumerate}
\item If $\cR(V,K)$ is separable, then the projective resonance ${\bf{R}}(V,K)$ is reduced and all its (linear) components are disjoint.
\item If moreover $\cR(V,K)$ is isotropic, then the projective resonance ${\bf{R}}(V,K)$ is reduced if and only if $\cR(V,K)$ is strongly isotropic.
\end{enumerate}
\end{thm}

\vskip 4pt

The main result of \cite{AFRS2} describes in an effective way the Hilbert series of any Koszul module whose resonance is strongly isotropic. Note that this result is known to hold only in characteristic zero.

\begin{thm}
\label{thm:main}
Let $V$ be an $n$-dimensional complex vector space and let $K\subseteq \bigwedge^2 V$ 
be a subspace such that the resonance $\cR(V,K)=\overline{V}^{\vee}_1\cup \cdots \cup \overline{V}^{\vee}_k$ is strongly isotropic. Then 
\[
 \dim W_q(V,K) = \sum_{t=1}^k \dim\, W_q(\overline{V}_t,0),\mbox{ for  all } \  q\geq n-3.
 \]
\end{thm}

In the case $\cR(V,K)=\{0\}$,  Theorem \ref{thm:main} reduces to Theorem \ref{thm:va-resonance}, that is,  to the sharp vanishing statement $W_q(V,K)=0$ for $q\geq n-3$. Theorem \ref{thm:main} can be regarded as a general \emph{Chen ranks formula} for Koszul modules generalizing and providing an effective version of  Suciu's Conjecture \cite{Su1} for the Chen ranks of hyperplane arrangement groups. 

\vskip 3pt
The proof of Theorem \ref{thm:main} is based first on the fact that if $\cR(V,K)=\overline{V}_1^{\vee}\cup \ldots \cup \overline{V}_k^{\vee}$ 
is strongly isotropic as in the Definition \ref{def:isotropic}, then the base locus ${\bf{B}}(V,K)$ defined by (\ref{eq:linearsection}) is a disjoint union of sub-Grassmannians 
\[
{\bf{B}}(V,K)=\mathbf{G}_1 \cup \cdots \cup \mathbf{G}_k,
\]
where $\mathbf{G}_t=\mbox{Gr}_2(\overline{V}_t^{\vee}) \subseteq \mathbf{G}$. One blows the Grasmannian $\mathbf{G}$ along base locus ${\bf{B}}(V,K)$ and on the blown-up Grassmannian $\widetilde{\mathbf{G}}$ one writes down a Koszul complex and an associated hypercohomology spectral sequence, broadly along the lines of the proof in \cite{AFPRW1} of Theorem \ref{thm:va-resonance}.  The formula in Theorem \ref{thm:main} appears as a result of having ultimately constructed 
 a natural morphism of graded $S$-modules
\begin{equation}
\label{eq:decomposition}
W(V,K)\longrightarrow \bigoplus_{t=1}^k W(\overline{V}_t,0)
\end{equation}
which turns out to be an isomorphism for degrees $q\geq n-3$.

\section{Chen ranks in geometry and topology}

We now discuss some of the connections between Koszul modules and topological invariants of groups. As explained in the Introduction, this has been one of the main reason for developing a theory of Koszul modules in the first place. For general references to this circle of ideas we refer to \cite{AFPRW2}, \cite{PS1}, \cite{PS2}.

\vskip 4pt

Let $G$ be a finitely generated group $G$. By taking successive commutators, one defines its \emph{lower central series}
$$G=\Gamma_1(G)\supseteq \Gamma_2(G)\supseteq \cdots \supseteq \Gamma_q(G)\supseteq \cdots,$$
where $\Gamma_{q+1}(G)=\bigl[\Gamma_q(G), \Gamma_1(G)\bigr]$. The \emph{lower central ranks} of $G$  are defined as the quantities 
$$\phi_q(G):=\mbox{rk } \mbox{gr}_q(G)=\mbox{rk }\Gamma_q(G)/\Gamma_{q+1}(G).$$
In general computing the ranks $\phi_q(G)$ is difficult and it is only in rare instances that commutative algebra, or algebraic geometry can contribute to directly determining these invariants.

\vskip 4pt

In order to make the problem more manageable, one passes instead to the \emph{metabelian quotient} of $G$. We denote by $G':=[G,G]$ the commutator subgroup of $G$ and set $G'':=[G',G']$. Then the \emph{metabelian quotient} of $G$ is the quotient group  $G/G''$,  fitting into the exact sequence 
\begin{equation}\label{eq:Alexander_inv}
1\longrightarrow G'/G''\longrightarrow G/G''\longrightarrow G/G'\longrightarrow 1,
\end{equation}
where we note that both groups $G/G'=:G_{\mathrm{ab}}$ and $G'/G''=:G'_{\mathrm{ab}}$ are abelian. 

\begin{definition}
The \emph{Chen invariants} of a group $G$ are defined as
 $$\theta_q(G):=\phi_q\bigl(G/G''\bigr).$$ 
\end{definition}

\begin{ex} 
One has surjective morphisms $\mbox{gr}_q(G)\twoheadrightarrow \mbox{gr}_q(G/G'')$, which is an isomorphism for $q\leq 3$. It follows that $\theta_q(G)\leq \phi_q(G)$, with equality when $q\leq 3$.
\end{ex}

The invariants $\theta_q(G)$ were introduced by Chen in \cite{Ch}, who later in \cite{Ch2} provided several ways to compute them in principle. To offer one way of understanding his work, we introduce the Alexander invariant, following \cite{PS1}:

\begin{definition}
The Alexander invariant of the group $G$ is defined as $B(G):=G'/G''$ regarded as a module over the group ring $\mathbb Z\bigl[G/G'\bigr]$, with the action being given by conjugation via the sequence (\ref{eq:Alexander_inv}).
\end{definition}

The Alexander invariant $B(G)$ admits a filtration with respect to the powers of the augmentation ideal
$I:=\mbox{Ker}\bigl\{\mbox{deg}\colon \mathbb Z\bigl[G/G'\bigr]\rightarrow \mathbb Z\bigr\}$. Massey \cite{Mas} showed that one has the following identification 
$$\mbox{gr}_q B(G)=I^q\cdot B(G)/I^{q+1}\cdot B(G)\cong \mbox{gr}_{q+2}\bigl(G/G''\bigr).$$

In particular, the Chen ranks of $G$ can also be computed via the Alexander invariant, which being a module over a commutative ring is more amenable to computations. The very definition of the Koszul module $W(G)$ of the group $G$ is an algebraization of the \emph{rational holonomy Lie algebra} of Chen \cite{Ch2}.

\vskip 4pt

Putting together work of Dimca--Hain--Papadima \cite{DHP}, Massey \cite{Mas} and Papadima--Suciu \cite{PS2}, one obtains a surjection $W(G)\twoheadrightarrow \mbox{gr } B(G)_{\mathbb C}$, which implies the inequality
\begin{equation}\label{eq:ineq_koszul}
\theta_{q+2}(G)=\mbox{rk } \mbox{gr}_q B(G)\leq \mbox{dim } W_q(G),
\end{equation}
with equality if the group $G$ is $1$-formal in the sense of Sullivan \cite{Sul}. Informally, a space $X$ is $1$-formal if its rational homotopy type is determined by its cohomology algebra. K\"aher groups, or hyperplane arrangement groups are all known to be $1$-formal, see again \cite{PS1} and references therein. Note that in this topological setting $G=\pi_1(X)$. 

\begin{ex}
Suppose $G=F_n$ is the free group on $n$ generators. Then 
$$\theta_q(F_n)=\mbox{dim } W_{q-2}\bigl(H_1(F_n,\mathbb C), 0\bigr).$$ Using (\ref{eq:Koszul_zero}), we find $\theta_q(F_n)=(q-1){n+q-2\choose q}$. This formula appears first in Chen's paper \cite{Ch}. 
\end{ex}

\subsection{Vanishing Chen invariants} Using the concepts introduced above, one has the following topological reformulation of Theorem \ref{thm:va-resonance}, see also \cite[Theorem 4.2]{AFPRW2}:
\begin{thm}\label{thm:fundamental}
Let $G$ be the fundamental group of a compact K\"ahler manifold with $n=b_1(G)$. If $G/G''$ is nilpotent, then $\theta_q(G)=0$, for $q\geq n-1$.
\end{thm}

Assume $X$ is a compact K\"ahler manifold and set $G=\pi_1(X)$. Then $H^1(G,\mathbb C)=H^1(X,\mathbb C)$ and the resonance $\cR(X):=\cR(\pi_1(X))$ does not vanish if and only if there exist non-proportional holomorphic forms $\omega_1, \omega_2\in H^0(X, \Omega_X)$ such that $\omega_1\cup \omega_2=0\in H^0(X, \Omega_X^2)$. By the Castelnuovo-de Franchis Theorem \cite{Cat} this is equivalent to the existence of a surjective map $f\colon X\twoheadrightarrow \Sigma_g$ over an algebraic curve of genus $g\geq 2$. In this case, one says that $X$ is \emph{fibred}. In fact, one has the following set-theoretic description of the resonance, see  \cite[Theorem C]{DPS}:
$$\cR(X)=\bigcup_{f:X\twoheadrightarrow \Sigma_C} f^*H^1(\Sigma_g,\mathbb C).$$
The resonance $\cR(X)$ is clearly linear and each such fibration $f\colon X\rightarrow \Sigma_g$ induces a surjection $f_*\colon \pi_1(X)\twoheadrightarrow \pi_1(\Sigma_g)$ and a $2g$-dimensional linear component $f^*H^1(\Sigma_g, \mathbb C)\subseteq H^1(X,\mathbb C)$ of the resonance variety.

\vskip 5pt
 
We record the following consequence of Theorem \ref{thm:fundamental}, see \cite[Theorem 4.16]{AFPRW2}:

\begin{cor}\label{cor:irr}
 Let $X$ be a non-fibred compact K\"ahler manifold. If $\pi_1(X)/\pi_1(X)''$ is nilpotent, then its nilpotency class is at most $2q(X)-2$, where 
 $q(X)=h^1(X, \OO_X)$ is the irregularity of $X$. \end{cor}
 
In particular, Corollary \ref{cor:irr} provides an upper bound for the Chen ranks of K\"ahler groups depending on the irregularity of the corresponding K\"ahler variety. Further important applications of Theorem \ref{thm:fundamental} concern the \emph{Torelli group} $T_g$ which measures from a homotopic-theoretic point of view the difference between the moduli space of curves and that of principally polarized abelian varieties. Denoting by $\mbox{Mod}_g$ the mapping class group of genus $g$, then $T_g$ is defined by the following exact sequence
$$1\longrightarrow T_g\longrightarrow \mbox{Mod}_g\longrightarrow \mbox{Sp}_{2g}(\mathbb Z)\longrightarrow 1.$$
 Due to fundamental work of Johnson \cite{J1}, \cite{J2}, it is known that one has an isomorphism 
$$H^1(T_g,\mathbb Q)\cong \bigwedge^3 H/H,$$
where $H:=H_1(\Sigma_g, \mathbb Q)$ is the first homology of a smooth algebraic curve of genus $g$. As explained in \cite{AFPRW2}, all the hypothesis of Theorem \ref{thm:fundamental} are satisfied and one concludes that for $g\geq 4$ the metabelian quotient $T_g/T_g''$ is nilpotent and its nilpotency class is at most 
$$b_1(T_g)-2={2g\choose 3}-2g-2.$$  
 
\begin{rmk}
Campana \cite{Cam} produced examples of K\"ahler groups of nilpotency class $2$. We are not aware of examples of higher nilpotency class, and constructing them is probably hard. On the other hand, Delzant \cite{Del} showed that solvable K\"ahler groups are virtually nilpotent. Finally, we mention the connection to a well-known conjecture of Koll\'ar \cite[Conjecture 4.16]{Ko} stating that the fundamental group of a smooth projective variety of Kodaira dimension zero is virtually abelian. This is known to hold in dimension at most $4$, see \cite{GKP}.  More generally, Campana \cite{Cam2} conjectured that the fundamental group of any \emph{special} smooth projective variety is virtually abelian.   
\end{rmk}

\subsection{Suciu's conjecture for hyperplane arrangement groups}
The fundamental groups of hyperplane arrangement groups represent a central class of groups for which the Chern ranks are particularly well understood and display deep connections to combinatorics. There are several good surveys on this topics, for instance  \cite{Su2}, \cite{Yuz}  and we shall concentrate on those recent developments that fit the other topics treated in this paper. 

\vskip 3pt

Let $\cA$ be an arrangement of hyperplanes in $\C^{m}$ and let 
\[
M({\cA}):= \C^{m}\setminus \bigcup_{H\in \cA} H
\] 
be the complement of the arrangement. We set $G(\cA):=\pi_1(M(\cA))$ and refer to this as the \emph{group of the arrangement}.

\vskip 3pt

The \emph{intersection lattice}\/ $L(\cA)$ of the arrangement is the set of al non-empty intersections of $\cA$ reverse ordered by inclusion. We also refer to $L(\cA)$ as the \emph{matroid} of $\cA$. Brieskorn and Arnold gave a description of the cohomology of $M(\cA)$ in terms of the lattice $L(\cA)$, which we now recall.  Let $E(\cA)$ be the exterior algebra over $\C$ generated in degree $1$ by elements $\{e_H\}_{ H\in \cA}$. We choose an ordering on the set of hyperplanes in $\cA$ and define a differential $\partial \colon E(\cA)\to E(\cA)$,  by setting 

\[
\partial \bigl(e_{j_1} \wedge \ldots \wedge  e_{j_s}\bigr)=\sum_{k=1}^s (-1)^{k-1} e_{j_1}\wedge\cdots 
\wedge\widehat{e}_{j_k}\wedge\cdots \wedge e_{j_s}.
\]

We denote by $I(\cA)$ the ideal of $E(\cA)$ generated 
by the elements of the form $\partial\bigl(e_{j_1}\wedge \cdots \wedge e_{j_s}\bigr)$, where $\mbox{codim}\bigl(H_{j_1}\cap \ldots \cap H_{j_s}\bigr)<s$. The \emph{Orlik--Solomon algebra} of $\cA$ is then the quotient $A(\cA):=E(\cA)/I(\cA)$.
As shown in \cite{OS}, one then has an isomorphism of graded algebras 
$$H^{\bullet}\bigl(M(\cA), \C\bigr) \cong A(\cA).$$

Note that from Brieskorn's work it also follows that $G(\cA)$ is a formal group. A consequence of this is that both sets of invariants $\phi_q\bigl(G(\cA)\bigr)$ and $\theta_q\bigl(G(\cA)\bigr)$ are combinatorially determined. A presentation of the fundamental group of $G(\cA)$ has been given in \cite{Sal}. On the other hand, Rybnikov \cite{Ryb} constructed examples of arrangements with isomorphic combinatorial structures, yet having non-isomorphic fundamental groups.

\begin{rmk}
A prominent example is given by the arrangement $\cA_{n-1}$ in $\mathbb C^n$ with hyperplanes $H_{ij}:\{z_i=z_j\}$. The complement $M(\cA_{n-1})$ is the configuration space of $n$ points on $\mathbb C$ and the fundamental group $\pi_1\bigl(M(\cA_{n-1})\bigr)=:P_n$ is the pure braid group on $n$ strands.
\end{rmk}

The components of the resonance variety were described by Falk--Yuzvinsky \cite{FY} in terms of elegant combinatorial structures called \emph{multinets}.

\begin{definition}
A {\em $k$-multinet} on an arrangement $\cA$ consists
of a partition $\mathcal{M}$ of $\cA$ into $k$ subsets $\cA_1,\ldots,\cA_k$,
together with an assignment of multiplicities,
$m\colon \cA\rightarrow  \mathbb Z_{>0}$, and a subset $\mathcal{X} \subseteq L_2(\cA)$,
called the \emph{base locus} of $\mathcal{M}$, such that the following 
are satisfied:
\begin{enumerate}
	\item 
	There is an integer $d$ such that $\sum_{H\in\cA_i} m_H=d$,
	for all $i=1, \ldots, k$.
	\item 
	For any two hyperplanes $H$ and $H'$ in different classes,
	$H\cap H'\in \mathcal{X}$.
	\item
	For each $X\in\mathcal{X}$, the sum
	$\sum_{H\in\cA_i\cap X\supset H} m_H$ is independent of $i$.
	\item \label{m4}
	For each $i=1, \ldots, k$, the space
	$\big(\bigcup_{H\in \cA_i} H\big) \setminus \mathcal{X}$ is connected.
\end{enumerate}
\end{definition}

This definition packages in combinatorial terms the notion of a pencil of plane curves with totally degenerate fibres, built out of an arrangement of lines in $\PP^2$. Given a multinet $\cM$ as above, for $i=1, \ldots, k$, we introduce the vectors $$u_i:=\sum_{H\in \cA_i} m_H\cdot e_H\in H^1\bigl(M(\cA),\mathbb C)$$ 
and set $\cP_{\cM}:=\langle u_2-u_1, \ldots, u_k-u_1\rangle \subseteq H^1(M(\cA), \mathbb C)$. It is shown in \cite[Theorem 2.4]{FY} that $\cP_{\cM}$ is an isotropic $(k-1)$-dimensional component of the resonance $\cR(\cA)$. Furthermore, if $\cB$ is a subarrangement of $\cA$ and $\cN$ is a multinet of $\cB$, then $\cP_{\cN}\subseteq H^1(M(\cA),\mathbb C)$ is still a component of $\cR(\cA)$. In fact all the components of the resonance appear in this way and one has the following set-theoretic equality
\begin{equation}\label{eq:res-arr}
\cR(\cA)=\bigcup_{\cB\subseteq \cA}   \ \ \bigcup_{\cN \mathrm{ multinet }\mbox{ }  \mathrm{on}\ \mbox{}\cB} \cP_N\subseteq H^1(M(\cA), \mathbb C).
\end{equation}
 
Components of $\cR(\cA)$ coming from multinets on $\cA$ itself are called \emph{essential}, the other, of the form $\cP_{\cN}$, where $\cB \subsetneq \cA$ are called \emph{non-essential}. The simplest multinets are provided by a flat $X$ of $\cA$ lying at the intersection of at least three hyperplanes in $\cA$. In this case, we introduce the subarrangement $\cA_X$ of all hyperplanes in $\cA$ containing $X$ and consider the multinet with $k=|\cA_X|$ blocks, each consisting of a single hyperplane, where we set all the multiplicities equal to be equal to $1$. The components $\cP_X$ of $\cR(\cA)$ arising in this way are called \emph{local components}. 

\begin{rmk}
The $(k,1)$-multinets on an arrangement $\cA$ correspond precisely to the local components of $\cR(\cA)$. Falk and Yuzvinsky \cite{FY} showed that non-local $(k, \ell)$-multinets can exist only when $k\in \{3,4\}$. Whereas there are plenty of examples of $(3,\ell)$-multinets for every $\ell$, there is only one known example of a $(4, \ell)$-multinet, namely the Hesse $(4,3)$-multinet obtained from the $9$ flex points of an elliptic curve. It is conjectured that this is the only $(4,\ell)$-multinet, see for instance Yuzvinsky's survey \cite[Conjecture 4.3]{Yuz}
\end{rmk}

\vskip 3pt

In this shown in \cite{CS} that all essential components of $\cR(\cA)$ are strongly isotropic. The same conclusion holds for all local components, respectively for all arrangements having at most triple points, see \cite[Proposition 9.8, Theorem 9.9 ]{AFRS2}. We have the following result from [AFRS2].

\begin{thm}\label{thm:Suciu_arr}
Let $\cA$ be an arrangement such that one of the following conditions hold. 
\begin{enumerate}
\item All components of $\cR(\cA)$ are either local or essential. 
\item $\cA$ has no $2$-flats of size greater than $3$.
\end{enumerate}
If $h_m$ is the number of components of $\cR(\cA)$ of 
dimension $m$, then 
\begin{equation}\label{eq:Chen_rk2}
\theta_q\bigl(G(\cA)\bigr) = (q-1)\cdot \sum_{m\geq 2} h_m\cdot {m+q-2\choose q}, \ \ \mbox{ for all }\ 
        q\geq \bigl|\cA\bigr|-1.
\end{equation}        
\end{thm}  

Suciu \cite{Su1} conjectured that the Chen rank formula (\ref{eq:Chen_rk2}) holds for \emph{every} hyperplane arrangement $\cA$ as long as $q\gg 0$. Theorem \ref{thm:Suciu_arr} is an optimal effective version of Suciu's Conjecture for arrangements having strongly isotropic resonance.

\vskip 5pt

As pointed out in \cite{AFRS2}, Theorem \ref{thm:Suciu_arr} can be applied to all graphic arrangements. Precisely, to a simple graph $\Gamma=(V, E)$ on a vertex set $V=\{1, \ldots, m\}$, we define the arrangement $\cA_{\Gamma}$ consisting of those hyperplanes $H_{ij}:\{z_i=z_j\}$ in $\mathbb C^m$, when $\{i,j\}\in E$. Denoting by $\kappa_r$ the set of (complete) $K_r$-subgraphs of $\Gamma$, then it is shown in \cite[Corollary 9.15]{AFRS2} that the Chen rank of group $G(\cA_{\Gamma})$ are given by the following formula, see also \cite{SS2}: 
$$\theta_q\bigl(G(\cA_{\Gamma})\bigr)=(q-1)(\kappa_3+\kappa_4), \mbox{ for all } \ q\geq \kappa_2-1.$$ 

\begin{question} 
Via the interpretation (\ref{eq:BGG}), we have that $\theta_q\bigl(G(\cA)\bigr)=\mbox{dim}_{\mathbb C} \mbox{Tor}^E_{q-1}\bigl(A(\cA),\mathbb C\bigr)_{-q}$. One can ask whether there is an explicit combinatorial formula, similar to Suciu's Conjecture (Theorem \ref{thm:Suciu_arr}) for the lower central series  ranks $\phi_q\bigl(G(\cA)\bigr)$.
It has been shown by Falk and Randell \cite{FR}, see also \cite[1.3]{SS1} that one has a formula 
\begin{equation}\label{eq:peeva}
\prod_{q=1}^{\infty} (1-t^q)^{\phi_q(G(\cA))}=\sum_{i=1}^{\infty}(-1)^i b_i\cdot t^i,
\end{equation}
where $b_i:=\mbox{dim } \mbox{Tor }_i^{A(\cA)}\bigl(\mathbb C, \mathbb C)_{-i}$, if and only if $A(\cA)$ is a Koszul algebra. This covers several important cases, see again \cite{SS1} and references therein, but Peeva \cite{Pe} showed that the formula (\ref{eq:peeva}) could not possibly hold for all arrangements $\cA$. This invites the question whether one can find a combinatorial formula for the ranks $\phi_q(\cA)$, at least for $q\gg 0$, in the same spirit as Suciu's Chen Ranks Conjecture?
\end{question}

\section{Koszul modules and Green's Conjecture}
We now discuss how, somewhat surprisingly, Koszul modules led to a new proof of Green's Conjecture on the syzygies of a  general canonical curve in characteristic zero and to a proof of this conjecture in positive characteristic. There are several good surveys on Green's Conjecture \cite{EL}, \cite{AF2} and \cite{F1} and we will try to take here a complementary view, focusing on the paper \cite{AFPRW1} and on the connections between syzygies and Koszul modules.

\vskip 4pt

To fix notation, we start with a projective variety $X$ defined over an algebraically closed field $\kk$ and with a globally generated line bundle $L$ on $X$. Set $r=r(L):=h^0(X,L)-1$ and denote by
$\varphi_L\colon X\rightarrow \PP^r=\PP H^0(X,L)^{\vee}$ the morphism induced by the linear system
$|L|$. We introduce the symmetric algebra $S:=\mbox{Sym } H^0(X,L)\cong \kk[x_0, \ldots, x_r]$. For every coherent sheaf $\F$ on $X$, we form the twisted coordinate $S$-module
$$\Gamma_X(\F,L):=\bigoplus_{q} H^0(X,\F\otimes L^{q}).$$

One writes down the twisted Koszul complex on $X$
$$
\begin{tikzcd}[column sep=23pt]
\bigwedge^{p-1} H^0(L)\otimes H^0(\F\otimes L^{q-1}) \ar[r, "d_{p+1,q-1}"] & \bigwedge^{p} H^0(L)\otimes H^0(\F \otimes L^q)  \ar[r, "d_{p,q}"] 
 &
 \bigwedge^{p-1} H^0(L)\otimes H^0(\F \otimes L^{q+1}),
\end{tikzcd}
$$
and define the \emph{Koszul cohomology group} $$K_{p,q}(X,\F,L):=\mbox{Ker}(d_{p,q})/\mbox{Im}(d_{p+1,q-1}).$$ 
Then, 
$$\mbox{dim } K_{p,q}(X,\F,L)=\mbox{dim } \mbox{Tor}_p^S\bigl(\Gamma_X(\F,L), \kk\bigr)_{p+q}=:b_{p,q}(X,\F,L)$$
are the dimensions of the spaces of $p$-syzygies of weight $q$ of the $S$-module $\Gamma_X(\F,L)$.

\vskip 4pt

When $\F=\OO_X$, then $\Gamma_X(L):=\Gamma_X(\OO_X,L)$ is the section ring of $X$ under the map $\varphi_L$. Then we set
$b_{p,q}(X,L):=b_{p,q}(X, \OO_X,L)$ and refer to these quantities as the \emph{Betti numbers}  of the pair $(X,L)$.

\vskip 4pt

As pointed out by Green \cite{Gr}, the Koszul cohomology groups are ordinary cohomology groups of certain syzygy vector bundles on $X$. Precisely, given $L$, we introduce the kernel bundle $$M_L:=\mbox{ker}\bigl\{\mathrm{ev}\colon H^0(X,L)\otimes \OO_X\longrightarrow L\bigr\}$$
 also considered in the sequence (\ref{eq:kernel_bundle}).
One has then the following canonical isomorphisms
\begin{equation}\label{eq:kosz_coh}
K_{p,q}(X, \mathcal{F}, L)  \cong \mathrm{Coker} \Bigl\{\bigwedge^{p+1}H^0(X,L) \otimes H^0(X,\mathcal{F} \otimes L^{q-1}) \longrightarrow H^0\bigl(X,\bigwedge^p M_L \otimes \mathcal{F} \otimes L^q
\bigr)\Bigr\},
\end{equation}
where the map in (\ref{eq:kosz_coh}) is the one induced from the short exact sequence
\begin{equation}\label{eq:wedge_seq}
0\longrightarrow \bigwedge^{p+1} M_L\longrightarrow \bigwedge^{p+1} H^0(X,L)\otimes \OO_X\longrightarrow \bigwedge^p M_L\otimes L\longrightarrow 0,
\end{equation}
after tensoring with the sheaf $\F\otimes L^{q-1}$ and taking cohomology. Computing Koszul cohomology of groups is generally difficult and for this reason in this paper we largely concentrate on the case when $X$ is a smooth curve.

\subsection{Green's Conjecture}
The research on the geometric theory of syzygies on curves in the last four decades has largely been inspired by three major conjectures appearing in the influential paper \cite{GL1} of Green and Lazarsfeld. They all concern the structure of the Betti diagram of a smooth curve 
$X\hookrightarrow \PP H^0(X,L)^{\vee}$ embedded by a very ample line bundle $L$. 

\vskip 4pt

The most famous of these conjectures is undoubtedly \emph{Green's Conjecture}, first posed in \cite{Gr}, which deals with smooth canonically embedded curves $X\subseteq \PP^{g-1}$, that is, with the case $L=\omega_X$. To be able to formulate it, we recall that the \emph{Clifford index} of the curve $X$ is defined as the following non-negative quantity 
$$
\mbox{Cliff}(X):=\mbox{min}\Bigl\{\mbox{deg}(L)-2h^0(X,L)+2: |L| \mbox{ is a linear system  with } h^0(X,L), h^1(X, L)\geq 2\Bigr\}.
$$

It is known that for a general $k$-gonal curve $X$ of genus $g$, one has $\mbox{Cliff}(X)=k-2$, whereas for a general curve  $\mbox{Cliff}(X)=\lfloor \frac{g-1}{2}\rfloor$. The Clifford index thus measures the \emph{complexity} of the curve in its moduli space, regarding hyperelliptic curves, that is, those for which $\mbox{Cliff}(X)=0$, as being the simplest curves of genus $g$.

\vskip 4pt

Green's Conjecture predicts the following equivalence for \emph{every} curve $X$:
\begin{equation}\label{eq:Green_Conj}
K_{p,2}(X,\omega_X)=0 \Longleftrightarrow p<\mbox{Cliff}(X). 
\end{equation}

\vskip 3pt

The statement $K_{p,2}(X, \omega_X)\neq 0$ if $p\geq \mbox{Cliff}(X)$ follows easily from the Green--Lazarsfeld \emph{Non-Vanishing Theorem} \cite{Gr}. The non-trivial implication is the vanishing $K_{p,2}(X,\omega_X)=0$, for all $p<\mathrm{Cliff}(X)$. The major challenge here lies in realizing the abstractly defined Clifford index  in the algebra of the canonical embedding of $X$, as the smallest index of a non-linear syzygy.

\vskip 3pt

Green's  Conjecture remains wide open for \emph{arbitrary} smooth curves. Voisin \cite{V1}, \cite{V2}  established Green's Conjecture for {\emph{general}} curves of every genus (in characteristic zero), using in an essential way the geometry of $K3$ surfaces. Her work has been essential for subsequent developments like Aprodu's Brill-Noether-theoretic sufficient condition for Green's Conjecture \cite{Ap}, the proof that all curves on $K3$ surfaces verify Green's Conjecture \cite{AF1}, or the work on the Prym--Green Conjecture \cite{FK1}, \cite{FK2}, \cite{FK3}. It is important to point out that for curves of maximal Clifford index $\mbox{Cliff}(X)=\lfloor \frac{g-1}{2}\rfloor$, the vanishing (\ref{eq:Green_Conj}) predicted by Green's Conjecture  is enough to determine entirely the Betti table of $X$, that is, the values of the Betti numbers $b_{p,q}(X, \omega_X)$ for all $p$ and $q$.  This is because via the description (\ref{eq:kosz_coh}) of the Koszul cohomology groups $K_{p,q}(X, \omega_X)$, the difference between the Betti diagrams on each diagonal are fixed, precisely 
\begin{equation}\label{eq:diff_Betti}
b_{p,1}(X, \omega_X)-b_{p-1,2}(X,\omega_X)=\frac{(g-2p-1)(g-p-1)}{p+1}{g-1\choose p-1},
\end{equation}
for $p\geq 1$ and $g\geq 3$. By Serre duality one also has via (\ref{eq:Green_Conj}) that $b_{p,2}(X,\omega_X)=b_{g-p-2,1}(X,\omega_X)$, for all $p$. This fact, together with the vanishing $b_{p,2}(X,\omega_X)=0$ for $p<\frac{g-1}{2}$ predicted by (\ref{eq:Green_Conj}), show that for each such curve we have $b_{p-1,2}(X,\omega_X)\cdot b_{p,1}(X,\omega_X)=0$, that is, the formulas (\ref{eq:diff_Betti}) determine all non-zero values of the Betti table of $(X, \omega_X)$.

\vskip 4pt

To make the discussion more concrete, we record the Betti table of a general canonical curve of odd genus $g=2i+3$:
\vskip 3pt

\begin{table}[htp!]
\begin{center}
\begin{tabular}{|c|c|c|c|c|c|c|c|c|}
\hline
$1$ & $2$ & $\ldots$ & $i-1$ & $i$ & $i+1$ & $i+2$  & $\ldots$ & $2i$\\
\hline
$b_{1,1}$ & $b_{2,1}$ & $\ldots$ & $b_{i-1,1}$ & $b_{i,1}$ & 0 & 0 &  $\ldots$ & 0 \\
\hline
$0$ &  $0$ & $\ldots$ & $0$ & $0$ & $b_{i+1,2}$ & $b_{i+2,2}$ & $\ldots$ & $b_{2i,2}$\\
\hline
\end{tabular}
\end{center}
\end{table}

The values of all non-zero Betti numbers  are computed by using the formulas (\ref{eq:diff_Betti}).

\subsection{Green's Conjecture via the tangent developable}  To prove Green's Generic Conjecture it suffices to exhibit a \emph{single} smooth curve $X$ of genus $g$ and Clifford index $\mbox{Cliff}(X)=\lfloor \frac{g-1}{2}\rfloor$ satisfying the condition 
\begin{equation}\label{eq:Green_suff}
K_{\lfloor \frac{g-3}{2}\rfloor,2}(X,\omega_X)=0.
\end{equation}

To  produce such a curve, even prior to Voisin's papers \cite{V1}, \cite{V2}, there has been a simple yet very appealing proposal of using the linear sections of the tangent developable variety to a rational normal curve in $\PP^g$. This proposal is documented in Eisenbud's paper \cite{Eis1} (who attributes it to Buchweitz--Schreyer) and we shall now describe it. With the novel perspective provided by Koszul modules, this strategy could finally be implemented in the paper \cite{AFPRW1}.  For the sake of simplicity, we shall concentrate on the case of characteristic zero, even though it is one of the main results of \cite{AFPRW1} that this argument can be carried out in positive characteristic as long as $p=\mbox{char}(\kk)\geq \frac{g+2}{2}$, therefore providing an answer to Green's Conjecture for generic curves in positive characteristic.

 \vskip 4pt
 
We start with a rational normal curve $\Gamma\subseteq \PP^g$ and consider its \emph{tangential variety} 
 $$\T\subseteq \PP^g.$$
 The surface $\T$ has a simple  affine parametrization 
 $$\mathbb A^2\longrightarrow \mathbb A^g, \ \ (x,t)\mapsto (x,x^2, \ldots, x^g)+t(1,2x,3x^2, \ldots, gx^{g-1}).$$
 Then $\T$ is a projectively normal surface of degree $2g-2$ with $\omega_{\T}\cong \OO_{\T}$ and in fact $\T$ can be regarded as a limit in $\PP^g$ of smooth $K3$ surfaces of genus $g$, though this fact plays no direct role in what follows. The construction of $\T$ follows by considering the jet bundle $\J=J\bigl(\OO_{\PP^1}(g)\bigr)$ on $\PP^1$. We denote by $U\cong \kk^2\cong H^0(\PP^1, \OO_{\PP^1}(1))$ and we identify $\mbox{Sym}^g U\cong H^0\bigl(\PP^1, \OO_{\PP^1}(g)\bigr)$. There is an exact sequence on $\PP^1$
 \begin{equation}\label{eq:jet_bundle}
 0\longrightarrow \mbox{Sym}^{g-2}U\otimes \OO_{\PP^1}(-2)\longrightarrow \mbox{Sym}^g U\otimes \OO_{\PP^1}\stackrel{\mathfrak{Ta}}\longrightarrow \J\longrightarrow 0, 
 \end{equation}
 where the Taylor morphism $\mathfrak{Ta}$ associates to a section its first order expansion around zero.  In particular, from (\ref{eq:jet_bundle}) we obtain an inclusion $\mbox{Sym}^g U\subseteq H^0(\PP^1, \J)\cong H^0\bigl(\PP(\J), \OO_{\PP(\J)}(1)\bigr)$. This inclusion induces a morphism: 
\begin{equation}\label{eq:norm}
\begin{tikzcd}[column sep=43pt]
\nu\colon \PP(\J)\ar[r, "|\mathrm{Sym}^g U|"]& \PP^g
\end{tikzcd}
\end{equation}
The tangential developable $\T$ is precisely the image of $\nu$ and this map is the normalization of $\T$. Furthermore, one also has an exact sequence of sheaves on $\T$
\begin{equation}\label{eq:normal}
0\longrightarrow \OO_{\T}\longrightarrow \nu_*\OO_{\PP(\J)}\longrightarrow \omega_{\Gamma}\longrightarrow 0.
\end{equation}

\vskip 4pt

The linear section of $\T$ are $g$-cuspidal rational curves $X:=\T\cap H\subseteq H\cong \PP^{g-1}$ and the $g$-cusps correspond to the points of intersection of the hyperplane $H$ with the rational normal curve $\Gamma$. It follows from the fundamental work of Eisenbud and Harris \cite{EH} on limit linear series that all $g$-cuspidal rational curves are general from the point of Brill--Noether theory, in particular any of their smoothing will be a curve of genus $g$ of maximal Clifford index $\lfloor \frac{g-1}{2}\rfloor$. This fact, which remarkably, holds \emph{irrespective} of the position of the $g$ cusps, was one of the reasons which made the tangential developable particularly attractive for trying to attack the Generic Green's Conjecture. Furthermore, since Koszul cohomology obeys a form of the Lefschetz Hyperplane Principle \cite[Theorem 3.7.b]{Gr}, one has isomorphism 
$$K_{p,q}(\T,\OO_{\T}(1))\cong K_{p,q}(X,\omega_X).$$
In order to establish (\ref{eq:Green_suff}), it suffices therefore to compute the Koszul cohomology groups $K_{p,2}(\T,\OO_{\T}(1))$. We now describe the strategy for carrying this out.  

\vskip 4pt

From the exact sequence (\ref{eq:normal}) we obtain after twisting and taking cohomology exact sequences
$$0\longrightarrow H^0(\T,\OO_{\T}(n)) \longrightarrow H^0\bigl(\T,\nu_*\OO_{\PP(\J)}(n)\bigr) \longrightarrow H^0(\Gamma, \omega_{\Gamma}(n)) \longrightarrow 0,$$
for each $n$. This sequence translates into an exact sequence of Koszul cohomology groups 

\begin{equation}\label{eq:Kosz_rel}
K_{i+1,1}\bigl(\T, \nu_* \OO_{\PP(\J)}, \OO_{\T}(1)\bigr) \longrightarrow K_{i+1,1}\bigl(\Gamma, \omega_{\Gamma}, \OO_{\Gamma}(1)\bigr)\longrightarrow K_{i,2}\bigl(\T, \OO_{\T}(1)\bigr)\longrightarrow 0,
\end{equation}
where the surjectivity of the last map in (\ref{eq:Kosz_rel}) follows because $K_{i,2}\bigl(\T, \nu_* \OO_{\PP(\J)}, \OO_{\T}(1)\bigr)=0$, see \cite[Propsition 5.6]{AFPRW1}.

\vskip 4pt

It is easy to show that using (\ref{eq:kosz_coh}), one has the following identification 
\begin{equation}\label{eq:id3}
K_{i+1,1}\bigl(\Gamma, \omega_{\Gamma}, \OO_{\Gamma}(1)\bigr)\cong \mathrm{Ker}\Bigl\{\mbox{Sym}^{i+2}U\otimes \bigwedge^{i+2} \mbox{Sym}^{g-1}U\longrightarrow \bigwedge^{i+2} \mbox{Sym}^g U\Bigr\}.
\end{equation}  
The map above is the one obtained after taking cohomology of a suitable twist of the exact sequence (\ref{eq:wedge_seq}) written on the rational curve $\Gamma$.

\vskip 4pt

Using the \emph{Kempf-Weyman geometric method} for constructing minimal free graded resolutions recalled in \cite[\S 4.3]{AFPRW1}, we then construct an explicit resolution of the $S$-graded ring $\bigoplus_{n\geq 0} H^0\bigl(\T, \nu_*\OO_{\PP(\J)}(n)\bigr)$, which in particular yields  that 
\begin{equation}\label{eq:id4}
K_{i+1,1}\bigl(\T,\nu_*\OO_{\PP(\J)}, \OO_{\T}(1)\bigr)=\mbox{Sym}^{2i+2}U\otimes \bigwedge^{i+2} \mbox{Sym}^{g-2}U.
\end{equation} 

\vskip 4pt

Understanding the maps in the presentation (\ref{eq:Kosz_rel}) via the identifications (\ref{eq:id3}) and (\ref{eq:id4}) looks challenging. The situation becomes transparent once we apply \emph{Hermite reciprocity} to both spaces appearing on the right hand sides of both (\ref{eq:id3}) and (\ref{eq:id4}).
Somewhat miraculously, after these identifications the sequence (\ref{eq:Kosz_rel}) becomes something recognizable as a particular Koszul module, which we introduce in what follows.

\begin{definition}\label{def:gordan}
We consider the \emph{Clebsch--Gordan decomposition} $$\bigwedge^2\mathrm{Sym}^{i+2}U\cong \mathrm{Sym}^{2i+2}U\oplus \mathrm{Sym}^{2i-2}U\oplus \mathrm{Sym}^{2i-6}U\oplus \cdots.$$ The \emph{Clebsch--Gordan} module is the Koszul module corresponding to this decomposition, that is,
$$W\bigl(\mathrm{Sym}^{i+2}U, \mathrm{Sym}^{2i+2}U\bigr).$$
\end{definition}
\vskip 4pt

Hermite reciprocity is a relatively elementary functorial isomorphism in characteristic zero
\begin{equation}\label{eq:hermite}
\mathfrak{He}\colon \mbox{Sym}^d\Bigl(\mbox{Sym}^iU\Bigr)\stackrel{\cong}\longrightarrow \bigwedge^i\Bigl(\mbox{Sym}^{d+i-1} U\Bigr),
\end{equation}
where we recall that $U\cong \kk^2$. As explained in \cite[Remark 3.4]{AFPRW1} Hermite reciprocity in characteristic zero can be regarded as a manifestation of the isomorphism $\mbox{Sym}^i (\PP^1)\cong \PP^i$. It is an important result of \cite[\S 3.4]{AFPRW1} that a characteristic free version of Hermite reciprocity is provided, which takes the form $\mbox{Sym}^d\bigl(D^i U)\cong \bigwedge^i \bigl(\mbox{Sym}^{d+i-1}U\bigr)$, where $D^iU$ is the space of  divided powers, which is isomorphic to $\mbox{Sym}^iU$ only in characteristic zero.

\vskip 5pt

Applying Hermite reciprocity to (\ref{eq:id3}), we have
$$K_{i+1,1}\bigl(\Gamma, \omega_{\Gamma}, \OO_{\Gamma}(1)\bigr)\cong \mbox{Ker} \Bigl\{\mathrm{Sym}^{i+2}U\otimes \mathrm{Sym}^{g-i-2}\bigl(\mathrm{Sym}^{i+2}U\bigr)\longrightarrow \mathrm{Sym}^{g-i-1}\bigr(\mathrm{Sym}^{i+2}U\bigr)\Bigr\}.$$
A priori it is not clear that the map appearing in this identification is the multiplication map $\delta_1$ appearing in the definition (\ref{eq:def-Wq}) of the Koszul module $W\bigl(\mathrm{Sym}^{i+2}U, \mathrm{Sym}^{2i+2}U\bigl)$, though this is eventually proven in \cite{AFPRW1}. Furthermore,  the identification  (\ref{eq:id4}) yields an isomorphism 
$$\mathfrak{he}\colon K_{i+1,1}\bigl(\T, \nu_*\OO_{\PP(\J)}, \OO_{\T}(1)\bigr)\stackrel{\cong}\longrightarrow \mathrm{Sym}^{2i+2}U\otimes \mathrm{Sym}^{g-i-3}\bigl(\mathrm{Sym}^{i+2}U\bigr).$$

\begin{thm}\label{thm:kosz_tangdev}
We have the following canonical identification of the Koszul cohomology of $\T$:
\begin{equation}\label{eq:canid}
K_{i,2}(\T, \OO_{\T}(1))\cong W_{g-3-i}\Bigl(\mathrm{Sym}^{i+2}U,  \mathrm{Sym}^{2i+2}U\Bigr).
\end{equation}
\end{thm}

In what follows, we explain identification (\ref{eq:canid}). We have the following diagram in which the first horizontal arrow is an isomorphism provided by Hermite reciprocity, whereas the second is an isomorphism onto the subspace $\mbox{Ker}(\delta_1)$:

\begin{equation}
\begin{aligned}
\xymatrix{
 K_{i+1}\bigl(\T,\nu_*\OO_{\PP(\J)}, \OO_{\T}(1)\bigr) \ar[d] \ar[rr]^{\mathrm{\mathfrak{He}}}  & \hspace{1.5cm} & \mathrm{Sym}^{2i+2}U\otimes \mbox{Sym}^{g-i-3}\bigl(\mathrm{Sym}^{i+2}U\bigr) \ar[d]_{\delta_2} \\
 K_{i+1}\bigl(\Gamma, \omega_{\Gamma}, \OO_{\Gamma}(1)\bigr) \ar[rr]^{\mathrm{\mathfrak{He}}}  \ar[d]& & \mathrm{Sym}^{i+2}U\otimes \mathrm{Sym}^{g-i-2}\bigl(\mathrm{Sym}^{i+2}U\bigr)\ar[d]_{\delta_1}\\
 K_{i,2}\bigl(\T, \OO_{\T}(1)\bigr) &  & \mathrm{Sym}^{g-i-1}\bigl(\mathrm{Sym}^{i+2}U\bigr)
}
\end{aligned}
\end{equation}

From this diagram, assuming the compatibility of the identifications provided via Hermite reciprocity between the Koszul cohomology groups in the left column, and via the projection of the Koszul complex on the right column respectively, one can define a morphism 
$$K_{i,2}\bigl(\T, \OO_{\T}(1)\bigr)\longrightarrow \frac{\mathrm{Ker}(\delta_1)}{\mathrm{Im}(\delta_2)}=:W_{g-i-3}\bigl(\mathrm{Sym}^{i+2}U, \mathrm{Sym}^{2i+2}U\bigr).$$
The existence of this map and the fact that this map is an isomorphism are the main results of \cite{AFPRW1}. The proof is essentially representation-theoretic.

\vskip 4pt

Note that Theorem \ref{thm:kosz_tangdev} displays some striking features. It turn out that for each $i\geq 1$, one Koszul module, namely $W\bigl(\mbox{Sym}^{i+2}U,  \mbox{Sym}^{2i+2}U\bigr)$ encodes in its various degree components the Koszul cohomology of general curves of different genera!
This phenomenon, already observed in \cite{Eis1} of mixing the genera of the syzygies of general curves in one algebraic object remains very mysterious.

\begin{question} 
Can one interpret the identification (\ref{eq:canid}) involving general curves of various genera as a geometric operation at the level of the stable moduli space of curves $\cM_{\infty}=\mbox{lim}_{g\rightarrow \infty} \cM_g$? When asking this question, we are guided by the numerous stabilization results in the cohomology of the moduli space of curves, in particular by the solution by Madsen and Weiss \cite{MW} of Mumford's Conjecture on the cohomology $H^*(\cM_{\infty}, \mathbb Q)$.
\end{question} 

Note that Theorem \ref{thm:kosz_tangdev} admits a \emph{characteristic free} interpretation given in \cite[Theorem 1.7]{AFPRW1} and which is ultimately responsible for the solution to the Generic Green's Conjecture in positive characteristic.

\begin{thm}\label{thm:generic_green}
The Generic Green's Conjecture holds for curves of genus $g\geq 3$ defined over an algebraically closed field $\kk$ of characteristic zero, or when $\mathrm{char}(\kk)\geq \frac{g+2}{2}$.
\end{thm} 
\begin{proof}
In the interest of simplicity we explain this only in characteristic zero. We use the identification (\ref{eq:canid}) coupled with Theorem \ref{thm:va-resonance}. The Koszul module $W\bigl(\mbox{Sym}^{i+2}U, \mbox{Sym}^{2i+2}U\bigr)$ has vanishing resonance, therefore by applying Theorem \ref{thm:va-resonance}, since $\mbox{dim } \mathrm{Sym}^{i+2} U=i+3$, we obtain that $W_{i}\bigl(\mbox{Sym}^{i+2}U, \mbox{Sym}^{2i+2}U\bigr)=0$. But taking $g=2i+3$, this implies via (\ref{eq:canid}) that 
$$K_{i,2}(\T, \OO_{\T}(1))=0,$$
which is precisely the vanishing (\ref{eq:Green_suff}) required for the Generic Green's Conjecture.

\vskip 4pt

If $g=2i+4$ is even, then (\ref{eq:Green_suff}) requires the vanishing $K_{i,2}(\T, \OO_{\T}(1))=0$. But
$$K_{i,2}(\T, \OO_{\T}(1))\cong W_{i+1,1}\bigl(\mbox{Sym}^{i+2}U, \mbox{Sym}^{2i+2}U\bigr)=0,$$
where the last vanishing is a consequence of the fact that for a Koszul module if $W_{q}(V,K)=0$, then also $W_{q+1}(V,K)=0$.
\end{proof}

Schreyer observed computationally that Green's Conjecture fails for every curve in low characteristic for infinitely many genera. The first counterexamples are $g=7$ and $\mbox{char}(\kk)=2$, respectively $g=9$ and $\mbox{char}(\kk)=3$. Eisenbud and Schreyer \cite{ES} conjectured that the Generic Green's Conjecture should hold whenever $\mbox{char}(\kk)\geq \frac{g-1}{2}$ and Theorem \ref{thm:generic_green} not only provides an alternative proof of the Generic Green's Conjecture in characteristic zero, but also (essentially) answers the Eisenbud--Schreyer Conjecture.

\begin{question}
Can one determine all characteristics $p=\mbox{char}(\kk)<\frac{g-1}{2}$, for which Green's Conjecture fails for every smooth curve $X$ of genus $g$? Is there a connection between this question and the results in \cite[Theorem 7.1]{RV}, where one establishes the surprising non-vanishing result $W_{2n-8}(V,K)\cong \kk$, for any $\kk$-vector space $V$ of dimension $n=3+p^a$, with $p=\mbox{char}(\kk)$, and for every $(2n-3)$-dimensional subspace $K\subseteq \bigwedge^2V$ such that $\cR(V,K)=0$?  Note that this connection, if it exists, cannot be provided by the tangential surface $\T$ used in \cite{AFPRW1}.
\end{question}

\begin{rmk}
As already mentioned, for an arbitrary Koszul module $W(V,K)$ the vanishing $W_q(V,K)=0$ trivially implies that $W_{q+1}(V,K)=0$. In this sense, the identification (\ref{eq:canid}) shows that the Generic Green's Conjecture in even genus is an immediate consequence of the Generic Green's Conjecture in odd genus, when indeed, the non-vanishing locus $$\Bigl\{[C]\in \cM_{2i+3}: K_{i,2}(C, \omega_C)\neq 0\Bigr\}$$ is a divisor in $\cM_{2i+3}$. This is in stark contrast to Voisin's original approach using curves on (smooth) K3 surfaces, when the even and the odd genus case required two largely independent papers \cite{V1} and \cite{V2}.
\end{rmk}  

\begin{question}
The \emph{Generic Prym--Green Conjecture} \cite{CEFS} predicts the shape of the resolution of a general line bundle $L\in \mbox{Pic}^{2g-2}(X)$ on a general curve $X$ of genus $g$. The conjecture remains unknown for even genus, in which case it reduces to one single vanishing statement
$$K_{\frac{g}{2}-2,1}(X,L)=0, \ \mbox{for a general line bundle } L\in \mathrm{Pic}^{2g-2}(X).$$ 
The original Prym--Green Conjecture has been formulated in \cite{CEFS} for line bundles $L=\omega_X\otimes \eta$, where $\eta\in \mbox{Pic}^0(X)[2]$ is an $\ell$-torsion point in the Jacobian of $X$. For the surprising failure of this conjecture when $g=8$ and $\ell=2$, we refer to \cite{CFVV}. 
The conjecture has been completely solved for odd genus, see \cite{FK2}, \cite{FK3}. 
 
\vskip 4pt
 
By analogy with \cite{AFPRW1} we can ask whether there is a degeneration o an embedded curve 
$$X\stackrel{|L|}\hookrightarrow \PP^{g-2}$$ of degree $2g-2$, such that for this degenerate curve the Prym--Green Conjecture turns into a statement on a particular Koszul module. What should this Koszul module be?
 \end{question}

\subsection{Latest developments} After the new approach via Koszul modules to the Generic Green Conjecture was put forward around 2019, several other proofs appeared soon after. Following the strategy of using the tangential developable, Park \cite{Pa} geometrized the part of Theorem \ref{thm:generic_green} that uses Hermite reciprocity and Koszul modules and thus produced a proof of the Generic Green's Conjecture (in characteristic zero), entirely in the realm of the Koszul cohomology of varieties. Kemeny \cite{K1} in a remarkable paper found a massive simplification of Voisin's strategy \cite{V1}, \cite{V2} relying on computing the Koszul cohomology of a polarized $K3$ surface $\bigl(S, \OO_S(H)\bigr)$ with $\mbox{Pic}(S)=\mathbb Z\cdot H$, where $H^2=2g-2$, and thus offering a direct proof of the Generic Green's Conjecture in characteristic zero.  

\vskip 3pt

Raicu and Sam \cite{RS} in a very interesting paper developed a bigraded version of the theory of Koszul modules and found a bigraded version of the vanishing Theorem \ref{thm:va-resonance}. Precisely, one considers two $\kk$-vector spaces $V_1$ and $V_2$, sets $V=V_1\oplus V_2$ and considers a subspace $K\subseteq V$. Then one defines a \emph{bigraded} Koszul module $W(V,K)$ over $S=\mbox{Sym}(V_1\oplus V_2)$, where the bidegree $(d,e)$-part of $S$ is equal to $\mbox{Sym}^{d}(V_1)\oplus \mbox{Sym}^e(V_2)$. In a way broadly reminiscent of Theorem \ref{thm:va-resonance}, Raicu and Sam show that 
\begin{equation}\label{eq:rs_equiv}
W_{\mathrm{dim}(V_2)-2, \mathrm{dim}(V_1)-2}(V,K)=0 \Longleftrightarrow K^{\perp}\subseteq V_1^{\vee}\otimes V_2^{\vee} \mbox{ contains no rank two tensors}.
\end{equation} 
The condition on the left hand side of (\ref{eq:rs_equiv}) can be rephrased as asking that the linear subspace $\PP(K^{\perp})\subseteq \PP\bigl(V_1^{\vee}\otimes V_2^{\vee}\bigr)$ be disjoint from the Segre variety $\PP(V_1^{\vee})\times \PP(V_2^{\vee})$. Compared to the single graded setting of Koszul modules in Theorem \ref{thm:va-resonance}, the role of the Grassmannian $\mathbf{G}=\mbox{Gr}_2(V^{\vee})$ is played by the Segre variety.  

\vskip 3pt
A striking application of the equivalence (\ref{eq:rs_equiv}) is a proof of the Generic Green Conjecture using \emph{ribbons}, which works when $\mbox{char}(k)=0$, or $\mbox{char}(\kk)\geq \frac{g-1}{2}$. Note that this is the same bound as the one in Theorem \ref{thm:generic_green}, except when $g=2p+1$, with $p$ being a prime and $\mbox{char}(\kk)=p$, in which case the bound in \cite{RS} is better by one than the one in \cite{AFPRW1}. Putting these results together, one has a full solution to the Eisenbud--Schreyer Conjecture \cite{ES}.

\begin{question} What is the topological counterpart of the equivalence (\ref{eq:rs_equiv}) in \cite{RS}? Is there a bigraded version of the theory of Chen invariants of topological objects?
\end{question}

\section{The Secant Conjecture, the Gonality Conjecture and related topics}

The Green--Lazarsfeld \emph{Secant Conjecture} \cite{GL1} is a generalization of Green's Conjecture to the case of non-special line bundles on a curve $X$. It predicts that as long as the degree of a line bundle $L\in \mbox{Pic}^d(X)$ is not too small, the syzygies of $L$ are accounted for by a uniform secant construction. Precisely, assuming that $L$ is non-special such that $$\mbox{deg}(L)\geq 2g+p+1-\mbox{Cliff}(X),$$ then $K_{p,2}(X,L)\neq 0$ if and only if $L$ is $(p+1)$-very ample, that is, every effective divisor of degree $p+2$ on $X$ imposes independent conditions on $|L|$.  If $d\geq 2g+p+1$, then Green \cite{Gr} showed that $K_{p,2}(X,L)=0$, so the conjecture automatically holds in this case. We may thus assume $d\leq 2g+p$.

Denoting for positive integers $a,b$ by $X_a-X_b\subseteq \mbox{Pic}^{a-b}(X)$ the \emph{difference variety} consisting of line bundle $\OO_X(D_a-E_b)$, where $D_a$ and $E_b$ are effective divisors of degree $a$ and respectively $b$ on $X$, the Secant Conjecture can be reformulated as the following equivalence for line bundles:  
\begin{equation}\label{eq:GL-secant}
K_{p,2}(X,L)\neq 0\Longleftrightarrow L-\omega_X\in  X_{p+2}-X_{2g-d+p}.
\end{equation}

\vskip 4pt 

Writing $\mbox{deg}(L)=2g+p+1-c$, the conjecture is known to hold for an arbitrary curve when $c=1$, see \cite{GL3}. Agostini \cite{Ag} recently proved the Secant Conjecture  when $c=2$ and $X$ is not bielliptic. Note that unlike the Green Conjecture, the Secant Conjecture depends both on a curve and on a line bundle on it, therefore there are various degree of genericity in which one can formulate it. The Generic Secant Conjecture is proved in \cite{FK1}, and we summarize the results:
\begin{thm}\label{thm:secant}
\hfill
\begin{enumerate}
\item The Secant Conjecture holds for a general curve $X$ of genus $g$ and a general line bundle $L\in \mathrm{Pic}^d(X)$.
\item The Secant Conjecture holds for \emph{every} smooth curve of odd genus $g$ and for \emph{every} line bundle $L\in \mathrm{Pic}^{2g}(X)$.
\end{enumerate}
\end{thm}

Another situation recently treated in \cite{F2} concerns the case when the condition appearing in the left hand side of (\ref{eq:GL-secant}) describes a divisor in the Jacobian variety $\mbox{Pic}^{d-2g+2}(X)$, which happens precisely when $d=g+2p+3$.  
 
\begin{thm}\label{thm:GL_div}
Let $X$ be a smooth curve of genus $g$, and fix $\lceil \frac{g-3}{2} \rceil\leq p \leq g-3$. Assume
$$\dim\ W^1_{p+2}(X)=2p-g+2.$$ One has the following equivalence for a line bundle $L\in \mathrm{Pic}^{g+2p+3}(X)$:
$$K_{p,2}(X,L)=0 \ \Longleftrightarrow \ L-\omega_X\notin X_{p+2}-X_{g-p-3}.$$
\end{thm}

A consequence of Theorem \ref{thm:GL_div} is that when $d=g+2p+3$, the Secant Conjecture holds for a general curve of genus $g$. 

\subsection{The Gonality Conjecture.}
From the Green--Lazarsfeld Non-Vanishing Theorem \cite{Gr}, it follows that for a non-special line bundle $L$ on $X$, one has $K_{r(L)-\mathrm{gon}(X), 1}(X,L)\neq 0$. The \emph{Gonality Conjecture} put forward by Green and Lazarsfeld in \cite{GL1} is the statement 
\begin{equation}\label{eq:gon_conj}
K_{h^0(L)-\mathrm{gon}(X), 1}(X,L)=0.
\end{equation}
The condition (\ref{eq:gon_conj}) is equivalent to the vanishing $K_{p,1}(X,L)=0$, for all $p\geq h^0(L)-\mathrm{gon}(X)$. The Gonality Conjecture therefore asserts that one can recognise the gonality of a curve in a precise way from the Betti diagram of any line bundle of sufficiently high degree.

\vskip 4pt

Using the interpretation of Koszul cohomology groups in terms of tautological bundles on symmetric products of curves, Ein--Lazarsfeld \cite{EL} showed that the Gonality Conjecture holds for \emph{every} curve $X$ and for \emph{every} line bundle of degree $\mbox{deg}(L)\gg 0$. Rathmann \cite{Rat} showed that (\ref{eq:gon_conj}) still holds when $\mbox{deg}(L)\geq 4g-3$, thus providing an effective version of (\ref{eq:gon_conj}). It has been showed in \cite[Theorem 0.1]{FK4}  that for a general $k$-gonal curve $X$ of genus $g$ one has that 
\begin{equation}\label{eq:generic_k}
K_{h^0(L)-k,1}(X,L)=0, \ \mbox{ for every } L\mbox{ with } \mbox{deg}(L)\geq 2g+k-1.
\end{equation} 
This degree, $2g+k-1$ is the smallest for which such a statement could hold. On each $k$-gonal curve $X$ there exist line bundles $L\in \mbox{Pic}^{2g+k-2}(X)$ for which the condition (\ref{eq:gon_conj}) does not hold.

\vskip 3pt

Niu--Park \cite{NP} solved the Effective Gonality Conjecture for arbitrary curves. Precisely, if $X$ is a $k$-gonal curve of genus $g$ which is not a smooth plane curve, then (\ref{eq:generic_k}) holds for every line bundle $L$ such that $\mbox{deg}(L)\geq 2g+k-1$. Their result is implied by a more general vanishing theorem, see \cite[Theorem 1.3]{NP}. Assuming $L$ and $B$ are line bundles on $X$ and $B$ is $p$-very ample, then if $h^1(X, L\otimes B^{\vee})\leq r(B)-p-1$, then $K_{p,1}(X,B,L)=0$. Just like in \cite{EL}, the proof of this result relies on interpreting the Koszul cohomology groups as geometric objects on symmetric products of $X$.

\subsection{The Betti diagram of a general $k$-gonal curve}

Let us consider a general $k$-gonal curve $X$ of genus $g$. From Green's Conjecture, which is known in this case \cite{Ap}, we have
$$K_{p,1}(X,\omega_X)=0 \ \ \mbox{ if and only if } \ \ \  p\geq g-k+1,$$
therefore determining the length of the linear of both rows of the resolution of $X$. Assume now that $X$ has non-maximal gonality, that is, $\mbox{gon}(X)\leq \frac{g+1}{2}$.  In this case, Green's Conjecture predicts the following resolution:

\begin{table}[htp!]
\begin{center}
\begin{tabular}{|c|c|c|c|c|c|c|c|c|c|}
\hline
$1$ & $2$ & $\ldots$ & $k-3$ & $k-2$  & $\ldots$ & $g-k$ & $g-k+1$ & $\ldots$ & $g-2$\\
\hline
$b_{1,1}$ & $b_{2,1}$ & $\ldots$ & $b_{k-3,1}$ & $b_{k-2,1}$ &  $\ldots$ & $b_{g-k,1}$ & $0$ &  $\ldots$ & $0$ \\
\hline
$0$ &  $0$ & $\ldots$ & $0$ & $b_{k-2,2}$ &  $\ldots$ & $b_{g-k,2}$ & $b_{g-k+1,2}$ & $\ldots$ & $b_{g-2,2}$\\
\hline
\end{tabular}
\end{center}
\end{table}

\vskip 4pt

We observe in this Betti table that $b_{p+1,1}(X,\omega_X)\cdot b_{p,2}(X,\omega_X)\neq 0$ for $k-2\leq p\leq g-k-1$. A general $k$-gonal curve $X$ is endowed with a cover $f\colon X\rightarrow \PP^1$ of degree $k$. This induces a $(k-1)$-dimensional scroll $Z\subseteq \PP^{g-1}$ containing the canonically embedded curve $X\subseteq \PP^{g-1}$ The Betti numbers of $(Z, \OO_Z(1))$ are well known, being computable via the Eagon-Northcott complex \cite{Sch1}. Since $X\subseteq Z\subseteq \PP^{g-1}$, it follows
\begin{equation}\label{gonericconj}
b_{p,1}(X,\omega_X)\geq b_{p,1}\bigl(Z,\OO_Z(1)\bigr)=p\cdot {g-k+1 \choose p+1}.
\end{equation}
One may have guessed that the inequality (\ref{gonericconj}) is an equality, but Bopp \cite{Bo} showed that for a general $5$-gonal curve of genus $g\gg 0$, if  $m:=\lceil \frac{g-1}{2}\rceil$, then $b_{m,1}(X,\omega_X)>b_{m,1}(Z, \OO_Z(1))$.

\begin{question}
What are the values $b_{p,1}(X,\omega_X)$ for a general $k$-gonal curve of genus $g$? It is shown in \cite[Theorem 0.4]{FK4} that 
$$b_{g-k,1}(X, \omega_X)=g-k.$$ This answers positively a conjecture of Schreyer's \cite{SSW}. For the Betti numbers $b_{p,1}(X, \omega_X)$ with $k-1\leq p\leq g-k-1$ there seems to be at the moment not even a prediction.
\end{question}

\subsection{Frobenius semistability and  Green's Conjecture in positive characteristic}
Suppose $X$ is a smooth algebraic curve over $\kk$, when $\mbox{char}(\kk)=p>0$. Set $\mathfrak{fr}\colon \kk\rightarrow \kk$ to be the Frobenius morphism and $X_1:=X\times_{\mathfrak{fr}} \mbox{Spec}(\kk)$. Intuitively, $X_1$ is the curve obtained from $X$ by raising all coefficients of the defining polynomials of $X$ to their $p$th power. We consider the relative Frobenius morphism 
$$F\colon X\longrightarrow X_1,$$
which is a cover of degree $p$. Following Raynaud \cite{Ray}, we define the bundle of \emph{locally exact differentials} on $X_1$ by the exact sequence
$$0\longrightarrow \OO_{X_1}\longrightarrow F_*\OO_X\longrightarrow B\longrightarrow 0.$$
Note that $B$ is a vector bundle of rank $p-1$ and via the \emph{Cartier operator} one constructs an anti-symmetric pairing $B\times B\rightarrow \omega_{X_1}$, in particular $B\cong B^{\vee}\otimes \omega_{X_1}$. The number $h^0(X_1,B)$ is called the \emph{a-invariant} of $X$ and satisfies the inequality $h^0(X_1,B)\leq g-f$, where $f$ is the \emph{p-rank} of the Jacobian of $X$. For a general curve $X$, we have $h^0(X_1,B)=0$. Ekedahl \cite{Ek} and Re \cite{Re} studied the Brill--Noether properties of $B$, for instance if $h^0(X_1,B)=g$, then $g\leq \frac{p(p-1)}{2}$.  We are however far from having a complete picture.

\vskip 4pt

It is proved in \cite[4.1.1]{Ray} that $B$ has a theta divisor in the sense of Raynaud, in particular it is stable. On the other hand, $B$ is not \emph{strongly stable}, that is, its Frobenius pullback destabilizes it and in fact $F^*(\omega_X)$ has a Harder--Narasimhan filtration
$$0\subseteq  B_p\subseteq B_{p-1}\subseteq \ldots \subseteq B_1=F^*B,$$
with $B_i/B_{i+1}\cong \omega_X^i$, for  $i=1, \ldots, p-1$.

\vskip 4pt

In a  related direction, Farkas--Larson \cite[Theorem 1.4]{FL} showed that the kernel bundle $M_{\omega_X}$ of a \emph{very general} curve of genus $g$ is strongly semistable, that is, the Frobenius pullbacks $(F^e)^*(M_{\omega_X})$ are semistable, for all $e\geq 1$. Note that even in genus $3$ it is clear that there exist smooth curves $X\subseteq \PP^2$ of genus $3$ for which $M_{\omega_X}$ is not strongly semistable, for instance the Fermat quartic in infinitely many characteristics, see \cite[Remark 3.8]{FL}. 

\begin{question}
Can one characterize in Brill--Noether-theoretic ways (also involving the vector bundle $B$) those curves $X$ for which the kernel bundle $M_{\omega_X}$ is strongly stable? Can one formulate in such terms a version of Green's Conjecture that should hold for all smooth curves of genus $g$?
\end{question}


\begin{thebibliography}{aaaaaaaa}
\bibitem[Al]{Al} J. Alexander,
{\em Topological invariants of knots and links},
Transactions of the American Math. Society \textbf{30} (1928), 275--306.
\bibitem[Ag]{Ag} D. Agostini, {\em{The Martens-Mumford Theorem and the Green-Lazarsfeld Secant Conjecture}},  Journal of Algebraic Geometry  \textbf{33} (2024), 629--654.
\bibitem[AFPRW1]{AFPRW1} M. Aprodu, G. Farkas, \c S. Papadima, C. Raicu,
and J. Weyman, {\em Koszul modules and Green's Conjecture},
Invent. Math. \textbf{218} (2019), 657--720.    
\bibitem[AFPRW2]{AFPRW2}
M. Aprodu, G. Farkas, \c S. Papadima, C. Raicu, and J. Weyman,
{\em Topological invariants of groups and {K}oszul modules}, Duke Math. J.
\textbf{171} (2022),  2013--2046.
\bibitem[AFRW]{AFRW}
M. Aprodu, G. Farkas, C. Raicu, and J. Weyman,
{\em Koszul modules with vanishing resonance in algebraic geometry},
Selecta Math. \textbf{30} (2024), Paper No.~24, 33 pp. 
\bibitem[AFRS1]{AFRS1}
M. Aprodu, G. Farkas, C. Raicu, and A. Suciu,
{\em Reduced resonance schemes and Chen ranks}, 
J. Reine Angew. Math. \textbf{814} (2024), 205--240.
\bibitem[AFRS2]{AFRS2}
M. Aprodu, G. Farkas, C. Raicu, and A. Suciu,
{\em{The effective Chen ranks conjecture}}, arXiv:2512.10160.
\bibitem[Ap]{Ap} M. Aprodu, {\em{Remarks on syzygies of $d$-gonal curves}}, Math. Research Letters \textbf{12} (2005), 387--400.
\bibitem[AF1]{AF1} M. Aprodu and G. Farkas, {\emph{The Green Conjecture for smooth curves lying on arbitrary $K3$ surfaces}}, Compositio Math. \textbf{147} (2011), 839--851.
\bibitem[AF2]{AF2} M. Aprodu and G. Farkas, {\emph{Koszul cohomology and applications to moduli}}, in: Grassmanians, Moduli Spaces and Vector Bundles (E. Previato ed.), Clay Mathematics Proceedings, Vol. 14, 2011, 25-50.
\bibitem[ACGH]{ACGH} E. Arbarello, M. Cornalba, P. A. Griffiths and J. Harris, {\emph{Geometry of algebraic curves}}, Volume I, Grundlehren der mathematischen Wissenschaften \textbf{267},
Springer-Verlag (1985).
\bibitem[ABS]{ABS} E. Arbarello, A. Bruno and E. Sernesi, {\emph{On hyperplane sections of $K3$ surfaces}}, Algebraic Geometry \textbf{5} (2017), 562--596.
\bibitem[Bo]{Bo} C. Bopp, {\emph{Syzygies of $5$-gonal curves}}, Documenta Math.  \textbf{20} (2015),  1055--1069. 
\bibitem[Cam1]{Cam} F. Campana, {\em Remarques sur les groupes de K\"ahler nilpotents}, Annales Sci. Ecole Normale Sup. \textbf{28} (1995), 307--316.
\bibitem[Cam2]{Cam2} F. Campana, {\em Orbifolds, special varieties and classification theory}, Annales Inst. Fourier
\textbf{54} (2004), 499--630.
\bibitem[Cat]{Cat} F. Catanese,  {\em Moduli and classification of irregular K\"ahler manifolds (and algebraic varieties) with Albanese general type fibrations}, Inventiones Math.
\textbf{104} (1991), 263--289.
\bibitem[Ch1]{Ch} K.~T.~Chen,
{\em Integration in free groups},
Annals of Math. \textbf{54} (1951), 147--162.
\bibitem[Ch2]{Ch2} K.~T.~Chen, {\em Extension of $C^{\infty}$ function algebra by integrals and Malcev completion of $\pi_1$}, Advances in Math. \textbf{23} (1977), 181--210.
\bibitem[CEFS]{CEFS} A. Chiodo, D. Eisenbud, G. Farkas and F.-O. Schreyer, {\em{Syzygies of torsion bundles and the geometry of the level $\ell$ modular varieties over $\overline{\mathcal{M}}_g$}}, Inventiones Math. \textbf{194} (2013), 73--118.
\bibitem[CS]{CS} D. Cohen and H. Schenck,
{\em Chen ranks and resonance},
Advances Math. \textbf{285} (2015), 1--27.  
\bibitem[CFVV]{CFVV} E. Colombo, G. Farkas, A. Verra and C. Voisin, \emph{Syzygies of Prym and paracanonical curves of genus $8$}, \'{E}pijournal G\'{e}om. Alg\'{e}brique \textbf{1} (2017), Art 23.

\bibitem[DHP]{DHP} A. Dimca, R. Hain, \c{S}. Papadima,
{\em The abelianization of the Johnson kernel},
Journal of the  European Math. Soc. \textbf{16} (2014), 805--822.
\bibitem[Del]{Del} T. Delzant, {\em L'invariant de Bieri-Neumann-Strebel des groupes fondamentaux des vari\'et\'es k\"ahleriennes}
Math. Annalen \textbf{348} (2010), 119--125.
\bibitem[DPS]{DPS} A. Dimca, \c{S}. Papadima, and A. Suciu, {\em Topology and geometry of cohomology jump loci},
Duke Math. Journal \textbf{148} (2009), 405--457.
\bibitem[EL1]{EL} L. Ein and R. Lazarsfeld, {\em{The gonality conjecture on syzygies of algebraic curves of
              large degree}}, Publ. Math. Inst. Hautes \'Etudes Sci. \textbf{122} (2015), 301--313.
\bibitem[EL2]{EL2} L. Ein and R. Lazarsfeld, {\em{Tangent developable surfaces and the equations defining algebraic curves}}, Bulletin Amer. Math. Soc. \textbf{57}  (2020), 23–-38. 
\bibitem[ENP]{ENP} L. Ein, W. Niu, and J. Park, {\em{Singularities and syzygies of secant varieties of nonsingular projective curves}}, Inventiones Math., \textbf{222} (2020), 615--665.                  
\bibitem[Eis1]{Eis1}
D. Eisenbud,
{\em Green's conjecture: An orientation for algebraists},
in: Free resolutions in commutative algebra and algebraic geometry (Sundance 1990), Res. Notes Math. \textbf{2}, Jones and Bartlett, Boston (1992), 51--78.
 \bibitem[Eis2]{Eis2} D. Eisenbud, {\em{The geometry of syzygies}}, Graduate Texts in Mathematics \textbf{229}, Springer-Verlag, New York, 2005.
\bibitem[EH]{EH} D. Eisenbud and J. Harris, {\em{Limit linear series: Basic theory}}, Inventiones Math. \textbf{85} (1986), 337-371.
\bibitem[ES]{ES}
D. Eisenbud and F.-O. Schreyer,
{\em Equations and syzygies of K3 carpets and unions of scrolls},  Acta Math. Vietnamica \textbf{44} (2019),  3--29.
\bibitem[Ek]{Ek} T. Ekedahl, {\em{  Torsten Ekedahl: On supersingular curves and abelian varieties}},
Math. Scandinavica \textbf{60} (1987), 151--178.           
\bibitem[FY]{FY} M. Falk and S. Yuzvinsky,
{\em Multinets, resonance varieties, and pencils of plane curves},
Compositio Math. \textbf{143} (2007), 1069--1088.
\bibitem[FR]{FR} M. Falk and R. Randell, {\em{The lower central series of a fibre-type arrangement}}, Invent. Math. \textbf{82} (1985), 77--88. 
\bibitem[F1]{F1} G. Farkas, {\em{Progress on syzygies of algebraic curves}}, Lecture Notes of the Unione Matematica Italiana 21 (Moduli of Curves, Guanajuato 2016) 107--138.
\bibitem[F2]{F2} G. Farkas, {\em{Difference varieties and the Green-Lazarsfeld Secant Conjecture}},  \'{E}pijournal de G\'{e}om\'{e}trie Alg\'{e}brique 2024 (special volume in honor of Claire Voisin), 11658.   
\bibitem[FL]{FL} G. Farkas and E. Larson, {\em{The Minimal Resolution Property for points on generic curves}}, Annales Sci. de l'Ecole Normale Sup\'erieure \textbf{58} (2025), 433--462.     
\bibitem[FMP]{FMP} G. Farkas, M. Musta\c{t}\u{a} and M. Popa, {\em{Divisors on
$\cM_{g, g+1}$ and the Minimal Resolution Conjecture for points on
canonical curves}}, Annales Sci. de L'\'Ecole  Normale Sup\'erieure \textbf{36} (2003), 553--581.
\bibitem[FK1]{FK1} G. Farkas and M. Kemeny, \emph{The generic Green-Lazarsfeld Conjecture}, Inventiones Math. \textbf{203} (2016), 265--301.
\bibitem[FK2]{FK2} G. Farkas and M. Kemeny, \emph{The Prym-Green Conjecture for line bundles of high order}, Duke Mathematical
Journal (2017), Duke Math. Journal \textbf{166} (2017), 1103--1124.
\bibitem[FK3]{FK3} G. Farkas and M. Kemeny, \emph{Minimal resolutions of paracanonical curves of odd genus}, Geometry \& Topology 22 (2018), 4235--4257.
\bibitem[FK4]{FK4}  G. Farkas and M. Kemeny, {\em{Linear Syzygies of Curves with prescribed gonality}}, Advances in Math. \textbf{356} (2019), 106810.
\bibitem[GKP]{GKP}  D. Greb, S. Kebekus, and T. Peternell, {\em Singular Spaces with Trivial Canonical
Class}, in: Minimal
Models and Extremal Rays, Advanced Studies in Pure Math. \textbf{70} (2016), 67--113.
\bibitem[Gr]{Gr} M. Green, {\em{Koszul cohomology and the cohomology of projective varieties}}, Journal of Differential Geometry \textbf{19} (1984), 125--171.
\bibitem[GL1]{GL1} M. Green and R. Lazarsfeld, {\em{On the projective normality of complete linear series on an algebraic curve}}, Inventiones Math. \textbf{83} (1986), 73--90.
\bibitem[GL2]{GL2} M. Green and R. Lazarsfeld,
\emph{Higher obstructions to deforming cohomology groups of line bundles}, 
J. Amer. Math. Soc. \textbf{4} (1991), 87--103. 
\bibitem[GL3]{GL3} M. Green and R. Lazarsfeld, {\em{Some results on the syzygies of finite sets and algebraic curves}}, Composition Math.\ \textbf{67} (1988), no.~3, 301--314.   
\bibitem[Hil]{Hil} D. Hilbert, {\em{\"Uber die Theorie der algebraischen Formen}}, Mathematische Annalen \textbf{36} (1890), 473--530.
\bibitem[HR]{HR} A. Hirschowitz and S. Ramanan, {\em{New evidence for Green's Conjecture on syzygies of canonical curves}}, Annales Sci. de l'\'Ecole Normale Sup\'erieure \textbf{31} (1998), 145--152.
\bibitem[J1]{J1} D. Johnson,
{\em The structure of the Torelli group {\rm I}: A finite set of generators for $\mathcal{I}$},
Annals of Math. \textbf{118} (1983), 423--442.
\bibitem[J2]{J2} D. Johnson,
{\em The structure of the Torelli group {\rm II}: A characterization of the group generated by twists on bounding curves},
Topology \textbf{24} (1985), 113--126.
\bibitem[K1]{K1} M. Kemeny, {\em{Projecting syzygies of curves}}, Algebraic Geometry \textbf{7} (2020),  561--580.
\bibitem[K2]{K2} M. Kemeny, {\em{Universal Secant Bundles and Syzygies of Canonical Curves}}, Inventiones Math. \textbf{223} (2021), 995--1026.
\bibitem[Ko]{Ko} J. Koll\'ar, {\em{Shafarevich Maps and Automorphic Forms}}, Princeton
University Press, 1995.
\bibitem[MW]{MW} I. Madsen and M. Weiss, {\em{The stable moduli space of Riemann
surfaces: Mumford’s conjecture}}, Annals of Math. \textbf{165} (2007), 843--941.
\bibitem[Mag]{Mag} W.~Magnus,
{\em \"{U}ber $n$-dimensionale Gittertransformationen},
Acta Math. \textbf{64} (1934), 353--367.
\bibitem[Mas]{Mas}  W.~Massey,
{\em Completion of link modules},
Duke Math. Journal \textbf{47} (1980), 399--420.
\bibitem[NP]{NP} W. Niu and J. Park, {\em{Effective gonality theorem of weight one-syzygies of algebraic curves}}, arXiv:2405.1344.
\bibitem[OS]{OS} P. Orlik and L. Solomon,
{\em Combinatorics and topology of complements of
hyperplanes}, Invent. Math. \textbf{56} (1980), 167--189.
\bibitem[PS1]{PS1} \c{S}. Papadima and A. Suciu, {\emph{Chen {L}ie algebras}}, Int. Math. Res. Not. \textbf{2004}
(2004),  1057--1086.
\bibitem[PS2]{PS2} \c S. Papadima and A. Suciu, {\emph{Vanishing resonance and representations of {L}ie algebras}},
J. Reine Angew. Math. \textbf{706} (2015), 83--101.
\bibitem[Pa]{Pa} J. Park, {\em{Syzygies of tangent developables surfaces and $K3$ carpets via secant varieties}}, Algebra \& Number Theory \textbf{19} (2025), 1029--1047. 
\bibitem[PPN]{PPN}  C. Pauly and A. Pe\'{o}n-Nieto,
{\em Very stable bundles and properness of the Hitchin map},
Geom. Dedicata \textbf{198} (2019), 143--148.  
\bibitem[Pe]{Pe} I. Peeva, {\em{Hyperplane arrangements and linear strands in resolution}}, Transactions Amer. Math. Soc. \textbf{355} (2002), 609--618.  
\bibitem[RS]{RS} C. Raicu and S. Sam, {\em{Bi-graded Koszul modules, $K3$ carpets, and Green's conjecture}}, Compositio Math. \textbf{158} (2022), 33--56.
\bibitem[RV]{RV} C. Raicu and K. VandeBogert, {\em{Stable sheaf cohomology on flag varieties}}, arXiv:2306.14282, to appear in the Journal American Math. Soc.
\bibitem[Rat]{Rat} J. Rathmann, {\em{An effective bound for the gonality conjecture}}, in: Varieties, Polyhedra, Computation, EMS Series of Congress Reports (2025), 567--578.
\bibitem[Ray]{Ray} M. Raynaud, {\em{Sections des fibr\'es vectoriels sur
une courbe}}, Bulletin de la Soci\'et\'e Math. de France \textbf{110} (1982), 103--125.  
\bibitem[Re]{Re} R. Re, {\em The rank of the Cartier operator and linear systems on curves}, J. Algebra \textbf{236} (2001), 80--92.  
\bibitem[Ryb]{Ryb} G.~Rybnikov, 
{\em On the fundamental group of the complement of a complex
hyperplane arrangement},  Funct. Anal. Appl.  
\textbf{45} (2011), 137--148. 
\bibitem[Sal]{Sal} M. Salvetti, \emph{Topology of the complement of real hyperplanes in $\mathbb C^n$}, Invent.
Math. \textbf{88} (1987), 603--618.
\bibitem[SS1]{SS1} H. Schenck and A. Suciu,
{\em Lower central series and free resolutions of hyperplane 
arrangements}, Transactions Amer. Math. Soc. \textbf{354} (2002), 
 3409--3433. 
\bibitem[SS2]{SS2} H. Schenck and A. Suciu,
{\em Resonance, linear syzygies, {C}hen groups, and the
{B}ernstein--{G}elfand--{G}elfand correspondence},
Transactions Amer. Math. Soc. \textbf{358} (2006),  2269--2289.
\bibitem[Sch1]{Sch1} F.-O. Schreyer, {\em{Syzygies of canonical curves and special linear series}}, Math. Annalen \textbf{275} (1986), 105-137.
\bibitem[SSW]{SSW} J. Schicho, F.-O. Schreyer and M. Weimann, {\em{Computational aspects of gonal maps and radical parametrization of curves}}, Appl.\ Algebra Engrg.\ Comm.\ Comput.\ \textbf{24} (2013), 313-341.
\bibitem[Su1]{Su1} A. Suciu, {\emph{Fundamental groups of line arrangements: enumerative aspects}},
in: {\em Advances in algebraic geometry motivated by physics} ({L}owell, {MA}, 2000),
43--79, Contemp. Math., vol. 276, Amer. Math. Soc., Providence, RI, 2001.
\bibitem[Su2]{Su2}  A. Suciu,
{\em Hyperplane arrangements and {M}ilnor fibrations},
Ann. Fac. Sci. Toulouse Math. \textbf{23} (2014), 417--481.
\bibitem[Sul]{Sul} D. Sullivan, 
{\emph{Infinitesimal computations in topology}}, Inst. Hautes \'Etudes Sci.
Publ. Math. (1977),  269--331.
\bibitem[V1]{V1} C. Voisin, {\em{Green's generic syzygy conjecture
for curves of even genus lying on a $K3$ surface}}, Journal of
European Math. Soc. \textbf{4} (2002), 363--404.
\bibitem[V2]{V2} C. Voisin, {\em{Green's canonical syzygy conjecture for generic curves of odd genus}},
Compositio Math. \textbf{141} (2005), 1163--1190.
\bibitem[Yuz]{Yuz} S. Yuzvisnky, {\em{Resonance varieties of arrangemenet complements}}, in: {\em Arangements of hyperplanes}, Advanced Studies in Pure Math. \textbf{62} (2012), 553--570. 
\bibitem[W]{W}
J. Wahl,
{\em The Jacobian algebra of a graded Gorenstein singularity},  Duke Math. J. \textbf{55} (1987), 843--871.

\end{thebibliography}
\end{document}